\documentclass[aos]{imsart}

\RequirePackage{amsthm,amsmath,amsfonts,amssymb,amscd,mathrsfs}
\RequirePackage[numbers]{natbib}
\RequirePackage[colorlinks,citecolor=blue,urlcolor=blue]{hyperref}

\startlocaldefs
\theoremstyle{plain}
\newtheorem{theorem}{Theorem}

\newtheorem{lemma}{Lemma}

\theoremstyle{remark}

\newtheorem{example}{Example}
\newtheorem{remark}{Remark}

\newtheorem{innercustomthm}{Condition}
\newenvironment{customcondition}[1]
  {\renewcommand\theinnercustomthm{#1}\innercustomthm}
  {\endinnercustomthm}

\newcommand{\cP}{\mathcal{P}}

\newcommand{\E}{\mathbb{E}}
\newcommand{\R}{\mathbb{R}}
\newcommand{\T}{\mathbb{T}}

\newcommand{\Z}{\mathbb{Z}}

\newcommand{\mH}{\mathcal{H}}

\newcommand{\mW}{\mathcal{W}}
\newcommand{\mX}{\mathcal{X}}
\newcommand{\mR}{\mathcal{R}}
\renewcommand{\L}{\mathscr{L}}

\newcommand{\eps}{\varepsilon}

\newenvironment{enumerate*}

\endlocaldefs

\begin{document}

\begin{frontmatter}
\title{Bayesian Nonparametric Inference in McKean-Vlasov models}
\runtitle{Bayesian inference for McKean-Vlasov models}

\begin{aug}
\author[A]{\fnms{Richard}~\snm{Nickl}\ead[label=e1]{nickl@maths.cam.ac.uk}},
\author[B]{\fnms{Grigorios A.}~\snm{Pavliotis}\ead[label=e2]{g.pavliotis@imperial.ac.uk}}
\and
\author[B]{\fnms{Kolyan}~\snm{Ray}\ead[label=e3]{kolyan.ray@imperial.ac.uk}}
\address[A]{Department of Pure Mathematics and Mathematical Statistics, University of Cambridge\printead[presep={,\ }]{e1}}

\address[B]{Department of Mathematics, Imperial College London\printead[presep={,\ }]{e2,e3}}
\end{aug}

\begin{abstract}
We consider nonparametric statistical inference on a periodic  interaction potential $W$ from noisy discrete space-time measurements of solutions $\rho=\rho_W$ of the nonlinear McKean-Vlasov equation, describing the probability density of the mean field limit of an interacting particle system. We show how Gaussian process priors assigned to $W$ give rise to posterior mean estimators that exhibit fast convergence rates for the implied estimated densities $\bar \rho$ towards $\rho_W$. We further show that if the initial condition $\phi$ is not too smooth and satisfies a standard deconvolvability condition, then one can consistently infer Sobolev-regular potentials $W$ at convergence rates $N^{-\theta}$ for appropriate $\theta>0$, where $N$ is the number of measurements. The exponent $\theta$ can be taken to approach $1/2$ as the regularity of $W$ increases corresponding to `near-parametric' models.
\end{abstract}



\end{frontmatter}


\section{Introduction}

We investigate the problem of conducting statistical inference in a class of dynamical systems whose state at time $t \in [0,T]$ is described by the probability density $\rho=\rho_W(t, \cdot)$ solving the nonlinear partial differential equation (PDE)
\begin{align}\label{PDE2}
\frac{\partial \rho}{\partial t}  &= \Delta \rho + \nabla \cdot (\rho \nabla (W \ast \rho)) \\
\rho(0, \cdot)&= \phi \notag,
\end{align}
known as the \textit{McKean-Vlasov equation}, where $\Delta, \nabla \cdot, \nabla$ denote the Laplacian, divergence and gradient operators, respectively. Furthermore, $\phi$ is the probability distribution of the initial state and $W$ is an interaction potential with Lipschitz gradient. In the present paper, we consider this equation on the $d$-dimensional torus $\T^d$, with periodic boundary conditions.
This PDE is a non-linear and non-local Fokker-Planck type equation and plays a fundamental role in a variety of scientific application areas  ranging from opinion dynamics \cite{hegselmannkrause} and other models in the social sciences~\cite{toscani2014} to mathematical biology 
\cite{painter_al_2024}, sampling and optimization~\cite{sprungk2023metropolisadjusted, Totzeck_al_2017}, plasma physics~\cite{Fournier_Jourdain_2017}, fluid mechanics~\cite{guillin2023uniform}, as well as statistical physics \cite{Ruffo_al_2018, frank04} and kinetic theory \cite{sznitman1991, Golse2016, MM13}. 

The problem we shall study here is how to recover the interaction potential $W$ from a discrete statistical measurement of the evolution of the macrosopic state $\rho$ of the system over time. We shall take a non-parametric approach where a possibly high- or infinite-dimensional model is postulated for the function $W$, arising from a Gaussian process prior $\Pi$, which is then updated via Bayes rule from `regression type' data $(Y_i, t_i, X_i)_{i=1}^N$ of the form
\begin{equation} \label{model}
Y_i = \rho_W (t_i, X_i) + \varepsilon_i,~~\varepsilon_i \sim^{iid} N(0,1),~~ i=1, \dots, N,
\end{equation}
where the $t_i, X_i$ are sampled discretely from the time-space cylinder $\mathcal X = [0,T] \times \T^d$. This approach is frequently taken in applications, e.g., \cite{Goodwin2020, computation9110119,lee2023learning}, and fits into the general paradigm of Bayesian non-linear inverse problems with PDEs \cite{Stuart2010},  which is amenable to an algorithmic treatment via MCMC and numerical PDE methods even in high- and infinite-dimensional models \cite{CRSW13, BGLFS17} -- see after (\ref{post}) below for more details. It also allows us to take advantage of recent theoretical developments in the field, e.g., \cite{N20, MNP21} and especially \cite{Nickl_2023}, but before we do so let us shed some more light on how such measurements may arise in concrete physical situations, explaining also the interpretation of the function $W$ as driving interactions.

A basic interacting particle model postulates the simultaneous evolution in time $t \in [0,T]$ of $n$ particles $X_t^i$ in a $d$-dimensional state space $\T^d$, solving the coupled system of $n$ stochastic differential equations (SDEs)
\begin{equation}\label{npart}
dX_t^i = -\frac{1}{n} \sum_{i \neq j}\nabla W(X_t^i-X_t^j)dt + \sqrt 2 dB_t^i,~~~~i=1, \dots, n,
\end{equation}
where the $B^i_t$ are independent $d$-dimensional Brownian motions, and where \textit{chaotic} initial conditions are assumed, i.e., the $X_0^i$ are all started independently, drawn randomly from the given initial distribution $\phi$. This setup can be considered as well in $\R^d$, in particular if we add a drift $-\nabla U(X^i_t)dt$ arising from a (known) confining potential $U$ to the  SDE (\ref{npart}). There has been a sequence of interesting recent papers trying to address the problem of inferring $W$ from measurements in models such as (\ref{npart}), see \cite{Amorino_al_2022, sharrock2023onlin, della2023lan, PavliotisZanoni2022, pavliotis2022method, Genon_laredo_2021, giesecke2019, amorino2024polynomial, Comte_al_2023, Hoffmann_al_2022} -- with earlier references being \cite{Bishwal_2011, Kasonga_1990}. Key statistical aspects of this task, such as dealing with the multi-dimensional setting $d>1$ and clarifying when the potential $W$ is indeed statistically identifiable, still remain broadly open, however. The important reference \cite{Hoffmann_al_2022} shows how the particle densities can be estimated empirically from the data, but as our results indicate, one can obtain faster convergence rates than in \cite{Hoffmann_al_2022} in important physically relevant settings (see below for more discussion). Furthermore, the hypotheses imposed in all the preceding references to identify $W$ are typically implicit and have not been verified for natural parameter spaces such as general Sobolev spaces. Note also that recently developed proof techniques for non-parametric inference in multi-dimensional diffusion models (see \cite{Nickl_Ray_2020, GiordanoRay2022, N24, HR22} and references therein) do not apply straightforwardly to the model (\ref{npart}) when the number $n$ of particles is large as the dimension $nd$ of the underlying state space then diverges rapidly.


To exploit the information present in measurements of such systems in the mean field limit ($n \to \infty$), one ought to take advantage of the rich mathematical theory that describes the macroscopic behaviour of particles interacting according to (\ref{npart}). In particular, it is well known that under mild assumptions on the interaction potential, and with chaotic initial data, the \textit{propagation of chaos} property holds-- we refer to \cite{sznitman1991} for a classical reference, to \cite{oelschlager1984, Malrieu2001} for particularly clear proofs, and \cite{guillin2023uniform, Lacker_LeFlem_2023, DGPS23} for important recent developments. Specifically, it is shown in \cite{Lacker_LeFlem_2023} that if one considers the \textit{marginal distributions} $P_t^{k,n} $ of any fixed number of $k=o(n)$ particles, then in the mean field limit $n \rightarrow \infty$ the laws $P_t^{k,n}$ are approximated in relative entropy by a \textit{product measure} $P^k_t$ whose coordinate distributions have  probability density $\rho_W$ solving the PDE (\ref{PDE2}). Moreover, $\rho=\rho_W$ is the density of the law $Law (X_t)$, where $X_t$ is the solution of the McKean SDE whose drift itself depends on the law of the process~\cite{rehmeier2024nonlinearmarkovprocessessense},
\begin{equation}\label{nlinsde}
dX_t = -(\nabla W \ast \rho )(t, X_t)dt + \sqrt 2 dB_t, ~t \ge 0,~~ X_0 \sim \phi.
\end{equation}
The existence of the solution to the last equation is not difficult to prove, in particular when the state space is the torus, by a fixed point argument in path space, see, e.g., \cite[Ch. I]{sznitman1991}. The intuition is that as $n$ increases, the distribution of any fixed particle is determined by the local dynamics of a Brownian motion with drift term reacting to the local density $\rho$ of nearby particles by convolution with the gradient of the interaction potential. 
Moreover, we can regard the laws of each distinct (fixed) pair chosen among $k=o(n)$ particles as \textit{approximately independent}. One can then record the `relative frequencies'  of the proportion of these $k$ particles spent in each of $N$ `bins' of a dissection of the time-space cylinder $[0,T] \times \T^d$. This allows one to obtain approximate (`histogram' type) measurements of the particle density $\rho_W$, and by the approximate independence of the $k$ marginals one can justify the independence assumption in model (\ref{model}). If we make the further simplifying assumption that the variance is homo-skedastic across the bins, the resulting model (\ref{model}) falls into the theoretical framework of \cite{Nickl_2023}. The assumption of constant noise variance could in principle be avoided by either directly studying a density estimation model where one has samples drawn i.i.d.~from $\rho_W$, or by minor adaptations of the developments in \cite{Nickl_2023}, but in order to focus on the main ideas and to expedite proofs, we refrain from this. We further believe that the results obtained here also strongly inform the `dependent' models (\ref{npart}), (\ref{nlinsde}), possibly even by using the non-asymptotic approximations in \cite{Lacker_LeFlem_2023} combined with results in \cite{R08} to bound the (one-sided) Le Cam distance for equivalence of statistical experiments. These questions are of independent interest and will be investigated elsewhere -- in the present article we content ourselves with aggregating the information in the data to a measurement (\ref{model}) of the particle densities. 


Once we have adopted the paradigm of Bayesian nonlinear regression with data (\ref{model}), the question of whether consistent statistical inference on $W, \rho_W$ is feasible can be determined by studying analytical properties of the parameter to solution map $W \mapsto \rho_W$ of the underlying PDE -- this was first shown in \cite{MNP21} and a general theory is laid out in \cite{Nickl_2023}. For the present non-linear Fokker-Planck equation (\ref{PDE2}) we exploit ideas from \cite{CGPS2020, Chazelle_al2017b}  that allow to realise $\rho_W$ as a limiting fixed point of sequences of suitable linear parabolic problems, as well as a `Girsanov type' estimate in relative entropy for solutions to linear and nonlinear Fokker-Planck equations
\cite{BRS_2016,Lacker_LeFlem_2023}; this bound permits us to replace the  arguments for linear PDEs from Sec.~2.1.1 in \cite{Nickl_2023}. Our key analytical results consist, first (Lemma \ref{lem:forward_lip}), of proving Lipschitz continuity of the non-linear forward map $W \mapsto \rho_W$ on bounded subsets of $C^2(\T^d)$:
$$\|\rho_{W_1} - \rho_{W_2}\|_{L^2(\mathcal X)} \lesssim \|\nabla (W_1-W_2) \ast \rho_{W_2}\|_{L^2(\mathcal X)} \lesssim \|W_1-W_2\|_{H^{-\beta}}$$ where $\beta>0$ is 
 a measure of the regularity of the initial condition $\phi \in H^\beta$. Using the theory from \cite{Nickl_2023} this implies that Gaussian process priors for interaction potentials $W$ give rise to sensible non-parametric models for the relevant regression functions $\rho_W$. In particular, this enables us to obtain fast convergence rates for recovery of $\rho_{W}$ by Bayesian plug-in estimates $\rho_{\bar W_N}$ arising from the posterior mean $\bar W_N = E^{\Pi}[W|(Y_i, t_i, X_i)_{i=1}^N]$. We then tackle the inverse problem and prove a partial converse to the first inequality in the last display, specifically a stability estimate
$$\|(W_1-W_2)\ast \rho_{W_2}\|_{L^2(\mathcal X)} \lesssim \|\rho_{W_1}- \rho_{W_2}\|_+ $$
for a norm $\|\cdot\|_+$ that involves partial derivatives in the space \textit{and} time variables, see (\ref{nlinbdfin}) for details. When $W_1$ has a bound on the number of non-zero Fourier modes, this further implies a novel injectivity (or `inverse stabilty') estimate for the `interaction potential to solution' map $W \mapsto \rho_W$ as soon as deconvolving the `factor' $\rho_{W_2}$ is possible. This permits us to provide concrete verifiable hypotheses on only the initial condition $\phi$ under which full inference on any Sobolev smooth interaction potential becomes feasible at polynomial convergence rates $N^{-\theta}, \theta>0$. The Bayesian posterior mean estimators that attain these convergence rates can be implemented by MCMC techniques, and unlike frequentist methods studied in earlier references, our approach does not require or postulate an explicit identification formula for $W$. Rather, the aforementioned stability estimates are sufficient to guarantee recovery, which is one of the main reason why Bayes methods are attractive in nonlinear inverse problems arising with PDEs and data assimiliation tasks, see \cite{Nickl_2023} and references therein.


Our proofs exploit information in the observations available at finite (in fact, short) time horizons, and are not based on any kind of ergodicity properties of the particle system.  
 Several of the recent papers on inference for the McKean SDE cited earlier develop inference methodologies based on observing particles over a long time interval, and rely on the ergodic properties of the dynamics in the limit as $T \rightarrow \infty$. This approach is not necessarily well-suited for the mean field dynamics. Indeed, it is well known~\cite{Dawson1983, CGPS2020, DGPS23,MGRGGP2020} that the McKean-Vlasov dynamics can exhibit phase transitions, i.e. multiple stationary states can exist. 
In particular, for the dynamics on the torus in the absence of a confining potential, the uniform measure is always a stationary state for the mean field dynamics, and clearly no information about the interaction potential can be deduced from it. For example, the spectral theoretic approach that was proposed in~\cite{PavliotisZanoni2022} and that is based on linearising the McKean-Vlasov operator around the invariant measure, does not work in this case, since the linearised Mckean-Vlasov operator is simply the Laplacian. But even when the McKean SDE has a unique \textit{informative} invariant measure, long term horizon asymptotics may not be useful as the invariant measures are generally smooth and hence not easily deconvolvable.

\section{Main results}\label{main}

\subsection{Notation}

We will denote by $C(\T^d)$ the space of (bounded) continuous functions on $\T^d$ while $L^2(\T^d)$ denotes the usual Hilbert space of square integrable functions for Lebesgue (probability) measure $dx$ on $\T^d$. The  spaces $H^\alpha(\T^d), C^\alpha(\T^d)$ consist of all functions that have partial (in the former case, weak) derivatives up to order $\alpha \in \mathbb N$ defining elements of $L^2(\T^d), C(\T^d)$, respectively, and we set $C^\infty(\T^d) = \cap_{\alpha>0} C^\alpha(\T^d)$, $C^0(\T^d)=C(\T^d)$ by convention. For regularity estimates for the PDE (\ref{PDE2}), it is also convenient to introduce the space $W^{\alpha, \infty}(\T^d)$ of functions in $C(\T^d)$ whose weak partial derivatives up to order $\alpha$ lie in the space $L^\infty(\T^d)$ of bounded functions. We have $C^\alpha(\T^d) \hookrightarrow W^{\alpha, \infty}(\T^d)$ where $\hookrightarrow$ means a norm-continuous imbedding. We can further define Sobolev spaces $H^\alpha(\T^d)=(H^{-\alpha}(\T^d))^*$ also for $\alpha<0$ as the topological dual space. The convolution $\ast$ of functions and measures is defined as usual, for instance if $f,g \in C(\T^d)$ then $f \ast g(x) = \int_{\T^d} f(x-y)g(y)dy, x \in \T^d$. 

We further define function spaces on the time-space cylinder $\mathcal{X} =[0,T] \times \T^d$, such as the Hilbert space $L^2(\mathcal X, \lambda)$ where $\lambda=\lambda_T$ is the uniform probability measure on $\mX$. For $B$ a normed space we also use standard parabolic PDE notation~\cite[Sec. 5.9]{evans} for the function spaces $L^p([0,T], B)$ of maps $H:\mathcal X \to \R$ whose norms $\|H(t,\cdot)\|_B$ lie in $L^p([0,T])$. The spaces $H^m([0,T], B)$ then denote those functions $H \in L^2([0,T], B)$ whose (weak) time derivatives $(\partial^j/\partial t) H$  for all $0 \le j \le m$ lie in $L^2([0,T], B)$, with corresponding Hilbert space norm. Similarly, we define the spaces $C^m([0,T] , B)$ of functions with time derivatives in $L^\infty([0,T], B)$. We also write $\mathcal{P}=\mathcal P(\T^d)$ for the set of all (Borel-) probability measures on $\T^d$. We will freely use standard results from real and Fourier analysis \cite{folland_1}.

\subsection{Gaussian process regression in the McKean-Vlasov model}

For an initial condition $\phi \in H^\beta$, consider the time marginal densities $\rho=\rho_W$ providing the unique periodic solutions to the non-linear parabolic PDE (\ref{PDE2}) on the time-space cylinder $\mathcal{X}  = [0,T] \times \T^d$, see Theorem \ref{thm:MV_well_posed}. Consider $N$ independent and identically distributed (iid) observations $(t_i,X_i,Y_i)_{i=1}^N$ arising from the model (\ref{model}), where the $(t_i, X_i)$ are drawn iid from the uniform distribution $\lambda$ on $[0,T] \times \T^d$, independently of the noise $\varepsilon_i$. We denote by $D_N = ((t_i,X_i,Y_i))_{i=1}^N$  the observation vector, with law $P_W^N$. As explained in the introduction, such regression-type forward models are widely used in (Bayesian) inverse problems \cite{Nickl_2023} and are approximately justified in interacting particle models (\ref{npart}) by the propagation of chaos phenomenon.

We devise a Bayesian model for the solutions $\rho_W$ of the PDE (\ref{PDE2}) in the non-linear regression framework (\ref{model}), placing a prior on the interaction potential $W$, our quantity of interest, while treating the initial condition $\phi$ as given. Another common approach in data assimilation with such time evolution equations would be to assign a prior to the initial condition $\phi$ -- see for instance the recent article \cite{NT23} in  the $2D$ Navier-Stokes model.

We employ Gaussian process priors \cite{RW06} for the interaction potential -- these will be supported on a separable normed linear space $(\mathcal W, \|\cdot\|_\mathcal W)$ satisfying the imbedding
\begin{equation}\label{suppw}
\mathcal W \hookrightarrow C^2(\T^d) \cap \left\{W: \int_{\T^d}W(x)dx =0\right\}.
\end{equation}
Note that the potential $W$ is identifiable only up to constants as the gradient $\nabla W$ relevant in models \eqref{PDE2},\eqref{npart},\eqref{nlinsde} annihilates constants, so fixing $\int W=0$ is natural for identifiability reasons. The separability is required to apply Gaussian process techniques from \cite{GN16, Nickl_2023} -- we can think of $\mathcal W$ as equal to a Sobolev space $H^{\alpha+1}(\T^d)$ for $\alpha>1+d/2$ or, to obtain sharper results, we can also take periodic Besov spaces $\mathcal W = B^\alpha_{\infty 1}(\T^d), \alpha \ge 2$ \cite[p. 370]{GN16}. Concrete examples will be discussed below, but for now we let $\Pi$ be any Borel probability measure on some space $\mathcal W$ satisfying (\ref{suppw}). We note that similar `$p$-exponential' priors could be used as well following ideas in \cite{ADH21, GiordanoRay2022, AW24}, but we do not pursue this here.

The posterior distribution obtained from such a prior $\Pi$ and data (\ref{model}) is given by Bayes' formula via standard arguments \cite{GvdV17, Nickl_2023};
\begin{equation}\label{post}
d\Pi(W|D_N) \propto e^{\ell_N(W)} d\Pi(W);~~\ell_N(W) = -\frac{1}{2} \sum_{i=1}^N |Y_i - \rho_W(t_i, X_i)|^2,~~W \in \mathcal W.
\end{equation}
Posterior draws $w \sim \Pi(\cdot|D_N)$ can be approximately calculated from Markov chain Monte Carlo (MCMC) techniques, for instance by the pCN, ULA or MALA algorithms, following the now well established paradigm of Bayesian inversion and uncertainty quantification in PDE models, see \cite{Stuart2010, CRSW13, BGLFS17, PSV_2022} and also \cite[Sec. 1.2.4]{Nickl_2023} for an overview. Each MCMC step requires one numerical solution of the non-linear Fokker-Planck equation (\ref{PDE2}), for which a variety of methods exists, see \cite{CCH15, SGGPUV2019, SS22} for recent references. We can then form ergodic averages $(1/K) \sum_{k=1}^K w_k$ of these MCMC draws $w_k$ to approximately compute the posterior mean estimate $\bar W = E^{\Pi}[W|D_N]$ of the interaction potential $W$, which in turn delivers estimates also of the particle densities $\rho_{\bar W}$ by one further numerical solution of the Fokker-Planck equation (\ref{PDE2}) with interaction potential $\bar W$. One can further use MCMC draws to perform Bayesian uncertainty quantification via posterior credible sets. For a concrete example and further discussion on applying MCMC methods to the Keller-Segel model for chemotaxis (which is a non-linear, non-local PDE of the form~\eqref{PDE2}, \cite[Sec. 6.5]{CGPS2020}), see for instance \cite{lee2023learning}.

One can also attempt to minimise the negative log-posterior over $\mathcal W$ (amounting to a MAP estimate, or Tikhonov regulariser), but we note that the problem is non-convex due to the nonlinearity of $W \mapsto \rho_W$, hence MCMC may be more robust to the presence of local optima in the criterion function. Computational guarantees for gradient based MCMC methods can be obtained in principle following ideas from \cite{NW24} combined with some of the stability estimates from the proofs that follow, but this is beyond the scope of the present paper. We shall focus here on analysing the statistical properties of the posterior distribution $\Pi(\cdot|D_N)$ under the objective (`frequentist') assumption that an actual ground truth potential $W_0$ generated the interacting particle system and thus data from equation (\ref{model}). Therefore the statistical analysis is under the law $P_{W_0}^N$ of the data vector $D_N$, and $\to^{P_{W_0}^N}$ will denote convergence in $P_{W_0}^N$-probability.

\subsection{Posterior contraction rates for the particle densities}

We establish contraction rates for posterior distributions (\ref{post}) arising from rescaled and possibly projected Gaussian process priors over $\T^d$. We start with a mean-zero Gaussian process $(V(x): x \in \T^d)$ with reproducing kernel Hilbert space (RKHS) $\mH$, whose law $\Pi_V$ satisfies the following assumption. For definitions and background material on Gaussian processes and their associated RKHS, see \cite[Ch. 11]{GvdV17} or \cite{RW06}. 

\begin{customcondition}{$\alpha$}\label{cond:RKHS}
Let $\Pi_V$ be a centred Gaussian probability measure on $\mathcal W$ from (\ref{suppw}) with RKHS $(\mH,\|\cdot\|_{\mH})$. Further suppose that for some $\alpha>d/2+1$, the continuous embedding $\mH \hookrightarrow H^{\alpha+1}(\T^d)$ holds. 
\end{customcondition}
To deal with various non-linearities in our regression problem, we follow ideas in \cite{MNP21} and rescale the `base prior' via
\begin{equation}\label{e:rescaled_GP}
W = \frac{V}{\sqrt{N}\delta_N};\qquad \qquad  \delta_N = N^{-\frac{\alpha+1+\beta}{2(\alpha+1)+2\beta+d}},
\end{equation}
to introduce extra regularisation in the posterior distribution. Here, $\alpha+1$ models the regularity of the interaction potential $W$ (so $\alpha$ models the regularity of $\nabla W$), while $\beta >0$ is determined by the smoothness of the initial condition $\rho_{W}(0, \cdot) = \phi \in H^\beta$. Among other things, this ensures space-time regularity estimates for $\rho_W$ corresponding to interaction potentials $W$ drawn from the posterior, see Theorem \ref{thm:PDE_lin} below. Note that such priors are special cases of the rescaled Gaussian process priors considered in several `direct' statistical settings \cite{GvdV17}.

We will further consider $L^2(\T^d)$-projections $\pi_{\mathcal W_N}(W)$ of the law of $W$ onto sequences of linear subspaces $\mathcal W_N \subseteq \mathcal W$, including the case $\mathcal W_N = \mathcal W$ for all $N$. The final prior law of such $W$ on $\mathcal W$ or $\mathcal W_N$ will be denoted by $\Pi=\Pi_N$. A typical example of such a projection is the truncation of an infinite dimensional Gaussian process onto its finite-dimensional counterpart, as is usually done in practice. We note that less aggressive shrinkage would be permitted in (\ref{e:rescaled_GP}) but we opt for the present choice as this allows us to directly use results from \cite{Nickl_2023} in the proofs.


The following theorem shows that we can solve the regression problem underlying (\ref{model}) via the implied prior arising from a Gaussian process model for the interaction potential $W$. We restrict to \textit{even} integers $\beta$ to facilitate the application of parabolic PDE theory in the proofs.

\begin{theorem}\label{thm:GP_forward}
Suppose $\phi \in H^\beta(\T^d)$ for some even integer $\beta \geq 3+d$ and let $\Pi_V$ be a Gaussian measure satisfying Condition \ref{cond:RKHS}. Consider the rescaled Gaussian process prior from \eqref{e:rescaled_GP} projected onto a linear subspace $\mathcal W_N$ of $\mathcal W$ such that $\|\pi_{\mathcal{W}_N}(w)\|_\mathcal{W} \leq C \|w\|_\mathcal{W}$ for all $w\in \mathcal{W}$ and some $C>0$ independent of $N$. Let $W_0 \in \mathcal H$ satisfy $\int_{\T^d} W_0 =0$, and assume further that there exists a sequence $W_{0,N}\in \mH \cap \mathcal W_N$ such that $\|W_{0,N}\|_{\mH} = O(1)$ and $\|W_0 - W_{0,N}\|_{H^{-\beta}} = O(\delta_N)$ as $N\to\infty$. Then for $M,L>0$ large enough, we have
$$\Pi(W \in \overline{\mathcal W}_{M,N} : \|\rho_W-\rho_{W_0}\|_{L^2(\mX)} \le L \delta_N|D_N) \to^{P^N_{W_0}} 1$$
as $N\to \infty$, where
\begin{equation}\label{WMN}
\overline{ \mW}_{M,N} = \left\{ W=W_1+W_2 \in \mathcal W_N: \|W_1\|_{H^{-\beta}} \leq M\delta_N, ~ \|W_2\|_{\mH} \leq M, ~ \|W\|_{C^2} \leq M \right\}.
\end{equation}
\end{theorem}

The convergence rate $\delta_N = N^{-\frac{\alpha+1+\beta}{2(\alpha+1)+2\beta+d}}$ derived above corresponds to the forward rate of an $(\alpha+1)$-smooth truth in an inverse problem with `polynomial ill-posedness' $\beta$. The self-convolutional term $\nabla W \ast \rho_W$ appearing in models \eqref{PDE2} and  \eqref{nlinsde} suggests the ill-posedness here should be related to the smoothness of $\rho_{W_0}$. In the present nonlinear parabolic PDE setting, the regularity of the densities $\rho_{W_0}$ is driven by that of the initial condition $\phi$ rather than that of $W_0$, see Theorem \ref{thm:PDE_lin}, which is reflected in the role of $\beta$ in the forward rate, see \eqref{eq:metric} below. Whether this is optimal is a delicate question concerning the optimality of regularity estimates in nonlinear parabolic PDEs. Note that the convergence rate $\delta_N$ accelerates to the `parametric' rate $1/\sqrt N$ as soon as either the model for $W$ or the initial condition $\phi$ is very smooth, i.e., as $\alpha$ or $\beta \to \infty$. 

To compare the convergence rates obtained in the preceding theorem to the existing literature, let us consider a general version of the interacting particle system \eqref{npart} on all of $\R^d$:
\begin{equation}\label{e:Hoffmann}
dX_t^i = - \nabla U(X^i_t)dt -\frac{1}{n} \sum_{i \neq j}\nabla W(X_t^i-X_t^j)dt + \sqrt 2 dB_t^i,~~~~i=1, \dots, n,
\end{equation}
where $U$ is a suitable confining potential. For (spatially) $\bar{\beta}$-smooth $\rho$  solving the corresponding non-linear Fokker-Planck equation, Della Maestra and Hoffmann \cite{Hoffmann_al_2022} show that the minimax rate for pointwise estimation of $\rho$ based on continuous observations (\ref{e:Hoffmann}) is the `usual' nonparametric rate $n^{-\bar{\beta}/(2 \bar{\beta}+d)}$. One can check that their minimax lower bound is established by considering $W=0$, that is, in a diffusion model without interaction term, where regularity of $U$ translates directly into that of $\rho$. However, when $U$ is zero or much smoother than $W$, our results show that we can expect convergence rates of the form $n^{-\frac{\alpha+1+\beta}{2(\alpha+1) +2\beta+d}}$ for the forward map $\rho_W$, corresponding to `additive' smoothness $\alpha+1+\beta$ (rather than just $\beta+1$ as might be implied by the smoothness of $\rho$ alone via Theorem \ref{thm:PDE_lin}) arising from the underlying convolution structure.

This suggests that methods that are only based on the smoothness of $\rho \in H^\beta$ and not (as we do via our Lemma \ref{lem:forward_lip}) on properties of the underlying parameter to solution map $(U,W) \mapsto \rho_{U,W}$ may not be able to fully exploit the information available in the data. This highlights a significant advantage of modelling $W$ directly, as the Bayes method does. We also note that by placing a prior on $W$, and hence $\rho_W$, one obtains posterior draws $\rho_W$ that lie in the range of the forward map (i.e., they solve the PDE (\ref{PDE2})). This is key for applying our stability estimate \eqref{stabilityforourworld} below for recovery of $W$. While the Bayes method does this automatically, it is not necessarily the case for other methods that first estimate $\rho_W$ directly by some smoothing method which relaxes the PDE constraint.

In summary, our results demonstrate that the estimation rates for $\rho$ are more complex than the regularity of $\rho$ alone might suggest, and that they depend on the interplay of the interaction term $W$ and the initial condition $\phi$ (as well as a nonzero confining potential $U$).

\subsubsection{Examples of Gaussian priors}

We now discuss examples of Gaussian priors satisfying the conditions of Theorem \ref{thm:GP_forward}.

\begin{example}[Periodized Mat\'ern process]\normalfont
Consider the Gaussian process on $\T^d$ given by
$$V(x) =  (2\pi)^{d/2}\sum_{0 \neq k\in\Z^d} \frac{1}{(1+4\pi^2|k|^2)^{(\alpha+1)/2}} g_k e_k(x), \qquad g_k \sim^{iid} N(0,1), \quad x\in \T^d,$$
where the $e_k \propto e^{2\pi i k (\cdot)}$ are the $L^2(\T^d)$-orthonormal trigonometric polynomials over $\T^d$. This is the $L^2(\T^d)$-series expansion resulting from periodising the usual Mat\'ern process on $\R^d$, see~\cite[Sec. A.1.1]{GiordanoRay2022} for details. We have excluded the constant function from the trigonometric basis to incorporate our identifiability condition $\int_{\T^d} W=0$.

By standard arguments (\cite[Ex.~2.6.15]{GN16}), the RKHS of $V$ equals 
\begin{equation}\label{e:H_def}
\mH := H^{\alpha+1}(\T^d) \cap \left\{\int_{\T^d} W=0\right\},
\end{equation}
with equivalent norm $\|\cdot\|_{\mH} \simeq \|\cdot\|_{H^{\alpha+1}}$. The above process has a version $V$ whose sample paths are in $C^{\alpha+1-d/2-\eta}(\T^d)$ for any $\eta>0$ (\cite[Prop. I.4]{GvdV17}). Thus using \cite[Lem I.4]{GvdV17}, we see that $V$ defines a Gaussian random element of $C^{r_0}(\T^d)$ for any $r_0<\alpha+1-d/2$, and since for $\alpha>1+d/2$ the space $C^{r_0}(\T^d)$ imbeds continuously into the separable Besov space $B^r_{\infty 1}(\T^d) \hookrightarrow C^2(\T^d), 2<r<r_0,$ (cf. \cite[p.371]{GN16}) we can realise the law of $V$ as a Gaussian Borel probability measure on $$\mathcal W = B^r_{\infty 1}(\T^d) \cap \Big\{\int_{\T^d} W =0\Big\}, ~\text{any } r<\alpha + 1 -d/2,$$ so that Condition \ref{cond:RKHS} is verified.

\begin{remark} \normalfont
Passing through a separable space is important to apply techniques from Gaussian process theory \cite[Ch. 2]{GN16}, but the introduction of Besov spaces could be avoided by increasing the smoothness to $\alpha>1+d$ and taking $\mathcal W = H^{\alpha-d/2}$ as in \cite[Thm B.1.3]{Nickl_2023}. 
\end{remark}

Now assume $W_0 \in \mH$ defined in~\eqref{e:H_def} which also lies in $\mathcal W$ (by what precedes, or abstractly by \cite[Cor. 2.6.17]{GN16}). Applying Theorem \ref{thm:GP_forward} with $\mathcal W_N =\mathcal W, W_{0,N}= W_0,$ then gives that for $M,L>0$ large enough,
$$\Pi\big(W\in \mathcal W_N:  \|\rho_W-\rho_{W_0}\|_{L^2(\mX)} \le LN^{-\frac{\alpha+1+\beta}{2(\alpha+1)+2\beta+d}}, ~ \|W\|_{C^2} \leq M |D_N\big) \to^{P^N_{W_0}} 1,$$
as $N\to\infty$ for this prior.
\end{example}

\begin{example}[Truncated Fourier prior]\label{ex:Fourier_series} \normalfont
For implementation by MCMC, and also in order to conduct inference on the interaction potential later, it is of interest to remove the higher Fourier frequencies from the Mat\'{e}rn prior for $W$ just constructed. For $K_N \in \mathbb N, K_N \to_{N \to \infty} \infty,$ consider the truncated Fourier series prior
\begin{equation}\label{e:truncated_fourier}
\pi_{\mathcal W_N}(W)(x) =\frac{1}{\sqrt N \delta_N} \sum_{k\in\Z^d: 0<|k| \leq K_N} \frac{1}{(1+|k|^2)^{(\alpha+1)/2}} g_k e_k(x), \qquad g_k \sim^{iid} N(0,1),
\end{equation}
which equals the $L^2$-projection of $W$ onto $\mathcal W_N = E_{K_N}$, the span of the trigonometric polynomials up to frequency $K_N$ (excluding constants). Let $W_0 \in \mH$, and set $W_{0,N} = \sum_{k\in\Z^d: 0<|k| \leq K_N} \langle W_0 , e_k\rangle_{L^2}  e_k \in \mH \cap \mathcal W_N $ to be the Fourier projection of $W_0$ onto the first $K_N$ frequencies. Then $\|W_{0,N}\|_{\mH} \leq \|W_0\|_{H^{\alpha+1}}$ and $\|W_0-W_{0,N}\|_{H^{-\beta}} \lesssim K_N^{-\alpha-1-\beta} =O(\delta_N)$ for $K_N \gtrsim N^\frac{1}{2(\alpha+1)+2\beta+d}$. Applying Theorem \ref{thm:GP_forward} gives that for $M,L>0$ large enough,
$$\Pi(W\in \overline{\mathcal W}_{M,N}:  \|\rho_W-\rho_{W_0}\|_{L^2(\mX)} \le LN^{-\frac{\alpha+1+\beta}{2(\alpha+1)+2\beta+d}} |D_N) \to^{P^N_{W_0}} 1,$$
as $N\to\infty$, with regularisation set $\overline{\mathcal W}_{M,N}$ given in (\ref{WMN}). [Note also that projection onto $E_{K_N}$ is continuous on $B_{\infty 1}^r(\T^d)$ with norm independent of $K_N$, in view of Remark 4.3.25 in \cite{GN16}.] The condition $K_N \gtrsim N^{\frac{1}{2(\alpha+1)+2\beta+d}}$ means the prior scaling rather than the truncation essentially determines the prior smoothness.
\end{example}

\begin{remark}[Symmetric potentials] \normalfont
In many physical applications, the interaction potential $W$ is symmetric. This can be encoded into the prior by separating the real and imaginary parts of the trigonometric polynomials $e_k \propto e^{2\pi i k (\cdot)}$ into $\phi_{2k}(x) \propto \cos (2\pi kx)$ and $\phi_{2k+1}(x) \propto \sin (2\pi kx)$, and  setting the prior coefficients of $\phi_{2k+1}$ to be zero. This yields prior and then also posterior draws that are symmetric about zero and all our results apply equally for Sobolev smooth \textit{symmetric} truths $W_0$.
\end{remark}

\begin{remark}[Almost finite-dimensional models] \label{fidi}\normalfont
One can also consider smoother Gaussian processes, for instance
$$V(x) =  (2\pi)^{d/2}\sum_{0 \neq k\in\Z^d} e^{-r|k|_1/2} g_k e_k(x), \qquad g_k \sim^{iid} N(0,1), \quad x\in \T^d,$$
where $|k|_1 = |k_1|+\dots+|k_d|$ and $r>0$. These model infinitely differentiable periodic functions with exponentially decaying Fourier coefficients. For example, as in \cite{rasmussen2023bayesian}, we consider the rescaled Gaussian process
\begin{equation*}
W = \frac{V}{\log N},
\end{equation*}
where less rescaling is needed than above since the sample paths of $V$ are already very regular. The RKHS of $V$ is
\begin{equation}\label{e:Hr}
\mH_r = \left\{ h = \sum_{0 \neq k\in\Z^d} h_k e_k: \|h\|_{\mH}^2:= \sum_{0 \neq k\in\Z^d} h_k^2 e^{r|k|_1} < \infty \right\},
\end{equation}
which embeds continuously into any Sobolev space $H^{\alpha+1}$, $\alpha>0$, meaning we have all the required regularity properties needed above. Our arguments can then be modified to extend to such settings by standard arguments, see \cite[Sec. C]{rasmussen2023bayesian} for details. For $W_0 \in \mH_r$, we thus have for $M,L>0$ large enough,
$$\Pi\big(W\in \mathcal W_N':  \|\rho_W-\rho_{W_0}\|_{L^2(\mX)} \le L\delta_N', ~ \|W\|_{W^{2,\infty}} \leq M |D_N\big) \to^{P^N_{W_0}} 1$$
as $N\to\infty$, where $\delta_N' = (\log N)^\eta/\sqrt{N}$ for some $\eta>0$ and $\mW_N' = \{W=W_1+W_2: \|W_1\|_{H^{-\beta}} \leq M \delta_N', \|W_2\|_{\mH_r} \leq M, \|W\|_{C^2} \leq M\}$. Such priors work for truths with exponentially decaying Fourier coefficients and are thus relevant when the interaction potential $W_0$ can be well approximated by relatively few Fourier modes. This occurs in several physical applications, for example the Kuramoto-Shinomoto-Sakaguchi ($O(2)$) model for synchronization or the (wrapped) Gaussian (attractive or repulsive) interaction potential; we refer to~\cite[Sec. 6]{CGPS2020} for further discussion and several examples.
\end{remark}

\subsection{Recovering the interaction potential}

We now consider conditions under which we can consistently recover the interaction potential $W_0$ from noisy observations (\ref{model}) of the McKean-Vlasov PDE \eqref{PDE2}. 


Note that the posterior contraction rate $\delta_N$ for $\|\rho_W - \rho_{W_0}\|_{L^2}$ from the previous theorem becomes close to the parametric rate $1/\sqrt N$ as $\beta \to \infty$, even if we keep the regularity $\alpha$ of the prior model for the potential $W$ fixed. This reflects the self-convolving nature of the McKean-Vlasov dynamics: the density $\rho_W$ can be very smooth even when $W$ is not, as long as the initial condition is sufficiently regular in a Sobolev sense. A consequence is that the interacting particle dynamics partially mask information about $W$ unless the initial condition is somehow atypical. Viewed this way, inferring $W$ becomes a statistical inverse problem which compounds a deconvolution step and a nonlinear inverse problem for the Fokker-Planck equation (\ref{PDE2}). We now show that this inverse problem can be solved, provided that the initial condition $\phi$ is not too smooth, measured by a lower bound on its Fourier coefficients. In other words, if we can prepare the system to start in a somewhat irregular initial state, then the observed dynamics of the interacting particle system will reveal the potential $W$ to the Bayesian posterior distribution arising from a prior that concentrates on at most a finite (but growing) number of Fourier modes. 

At the heart of the main statistical result Theorem \ref{thm:deconvolution} is the following stability estimate for the forward map $W \mapsto \rho_W$ under such hypotheses.

\begin{theorem}\label{stabest}
Suppose $W_0\in H^{\alpha+1}$ for $\alpha>1+d/2$ and $\phi \in H^\beta(\T^d)$ where $\beta \geq 3+d$ is an even integer. Assume $\phi$ is strictly positive $\phi \ge \phi_{min}>0$ on $\T^d$ and that the Fourier coefficients $\hat{\phi}$ of $\phi$ satisfy 
\begin{equation}\label{decon}
|\hat \phi(k)| \ge c |k|^{-\zeta},  ~  k \in \Z^d,~~\text{some } \zeta>\beta+d/2.
\end{equation}
Let $W = \sum_{k\in\Z^d: |k| \leq K} W_k e_k \in E_K$ be band-limited by $K$ Fourier frequencies and assume $\|W\|_{W^{2,\infty}} + \|W_0\|_{W^{2,\infty}} \le M$ as well as $\int_{\T^d} (W-W_0)dx=0$. Then there exists a constant $C=C(d,T,M,\|\phi\|_{H^\beta}, \phi_{min})>0$ such that for any $0<t_0 \leq \frac{c}{2C} K^{-\zeta}$ we have
\begin{equation}\label{stabilityforourworld}
\begin{split}
& \|W - W_0\|_{L^2}^2 \\
& \leq C \left(  K^{-2\alpha-2} + t_0^{-1}K^{2\zeta}  \big(\|\rho_{W}- \rho_{W_0}\|^2_{H^1([0,t_0], H^{-1})} + \|\rho_{W}- \rho_{W_0}\|^2_{L^2([0,t_0], H^1 )}\big) \right).
\end{split}
\end{equation}
\end{theorem}

The analytical intuition of this stability estimate could be phrased as saying that if we can obtain precise estimates of the space and time derivatives of $\rho_{W_0}$, we will be able to identify $W_0 \in \mH$
under the hypothesis (\ref{decon}). So \textit{both} the space and time dynamics of the densities $\rho_W$ are relevant for reconstruction. We can combine this estimate with the truncated Fourier series prior in Example \ref{ex:Fourier_series} to consistently recover the underlying potential. 

\begin{theorem}\label{thm:deconvolution}
Suppose $W_0 \in {H^{\alpha+1}}(\T^d)$ for $\alpha>1+d/2$ satisfies $\int W_0 =0$ and let $\phi \in H^\beta(\T^d)$, where $\beta \geq 3+d$ is an even integer. Further assume that $\phi$ is lower bounded by $\phi_{min}>0$ and satisfies (\ref{decon}). Consider the truncated Fourier series prior in Example \ref{ex:Fourier_series} with $K_N \simeq N^{\frac{1}{2(\alpha+1)+2\beta+d}}$. Then for $L>0$ large enough,
$$\Pi(W\in E_{K_N}: \|W-W_0\|_{L^2} \geq LK_N^{3\zeta/2} \delta_N^{\frac{\beta-2}{\beta}}  |D_N) \to^{P^N_{W_0}} 0,$$
as $N\to\infty$, where $\delta_N = N^{-\frac{\alpha+1+\beta}{2(\alpha+1)+2\beta+d}}$. In particular, the convergence rate is $LN^{-\theta}$ for
\begin{equation}\label{e:theta}
\theta = \frac{(\alpha+1+\beta)(\beta-2)/\beta - 3\zeta/2}{2(\alpha+1)+2\beta+d}.
\end{equation}
Moreover, if $\bar W_N = E^\Pi[W|D_N] \in E_{K_N}$ is the posterior mean vector, then we also have $$\|\bar W_N - W_0\|_{L^2} = O_{P_{W_0}^N}(N^{-\theta}) \text{ as } N \to \infty.$$ 
\end{theorem}

Theorem \ref{thm:deconvolution} uses the stability estimate \eqref{stabilityforourworld} with $t_0 \simeq K_N^{-\zeta} \to 0$ small, which shows that as long as the initial condition $\phi$ is sufficiently irregular, the \textit{short-time} dynamics of the system are already enough to consistently recover $W$ as $N\to \infty$. This is especially relevant in our setting where ergodic averages may be uninformative for $W$, since the uniform distribution is always a stationary state for the mean-field dynamics on the torus~\cite{CGPS2020}. Whether one can further exploit long-time information under relevant physical scenarios is a subtle question given the existence of phase transitions and, consequently, of multiple stationary states for McKean-Vlasov dynamics. Similar remarks apply to `intermediate' time horizons due to smoothing properties of the underlying non-linear semi-group. In any case, our result already shows that Bayesian methods can extract relevant short-time information from the data.

The condition $\zeta >\beta+d/2$ ensures that $\phi$ can satisfy both $\phi\in H^\beta$ and \eqref{decon} simultaneously. If the lower bound \eqref{decon} is too weak, i.e. the initial density is much harder to deconvolve than its Sobolev smoothness $\beta$ suggests, then our result may not imply consistency. Consistency with polynomial rates is possible as soon as $\theta>0$, which holds for $\zeta$ satisfying
$$\beta + d/2 < \zeta < \frac{2(\alpha+1+\beta)(\beta-2)}{3\beta}.$$
This range is non-empty for $\alpha>0$ large enough. For semi-parametric applications to Bernstein-von Mises theorems (e.g., as in~\cite[Sec.~4.1.1]{Nickl_2023}) it is important that these rates can be sufficiently fast in regular models: for instance we can always obtain $3 \theta >1$ as long as $\alpha \ge \alpha_0(\beta,d)$ for some $\alpha_0=\alpha_0(\beta, d)$, if (\ref{decon}) holds for $\zeta=\beta+d/2+\eta$ and all $\eta>0$.

One can further check that the rate $K_N^{3\zeta/2} \delta_N^{\frac{\beta-2}{\beta}}$ holds in Theorem \ref{thm:deconvolution} if one takes any $K_N \gtrsim N^{\frac{1}{2(\alpha+1)+2\beta+d}}$, though taking larger $K_N$ here deteriorates the final rate. One can also extend the theorem to $K_N \lesssim N^\frac{1}{2(\alpha+1)+2\beta+d}$ giving `deconvolution type' (see, e.g., \cite[p.458]{GN16}) convergence rate
$$K_N^{-\alpha-1} + K_N^{3\zeta/2} \left( \xi_N^{\frac{\beta-2}{\beta}} + \xi_N^\frac{\beta}{\beta+1} \right),$$
where now $\xi_N = \delta_N + K_N^{-\alpha-1-\beta}$ is a possibly slower contraction rate (since the prior truncation may now determine the `forward' contraction rate from Theorem \ref{thm:GP_forward}).

\begin{remark}[Deconvolution rates for almost finite-dimensional models] \normalfont
We can also consider the truncated version of the prior in Remark \ref{fidi}, truncating at level $K_N$ as in \eqref{e:truncated_fourier}. For $W_0 \in \mH_r$ defined in \eqref{e:Hr} and $K_N \geq \frac{1}{r}\log N$, we obtain the same near-parametric contraction rate $\delta_N' = (\log N)^\eta /\sqrt{N}$ for the densities $\rho_W$ as in Remark \ref{fidi}. Turning to the deconvolution, we deduce the rate
$$\|W-W_0\|_{L^2}^2 \lesssim e^{-rK_N} + K_N^{3\zeta} (\delta_N')^{\frac{2(\beta-2)}{\beta}} \lesssim (\log N)^{3\zeta + \frac{2(\beta-2)\eta}{\beta}} N^{-\frac{\beta-2}{\beta}}.$$
Thus for $W_0 \in \mH_r$, one can also recover the true potential $W_0$, and thus solve the inverse problem, at the near-parametric rate $(\log N)^{\eta}/\sqrt{N}$ as the regularity $\beta \to \infty$.
\end{remark}


We next give a concrete example of a distribution $\phi$ for which the above theorem applies and consistent recovery of the potential $W_0$ with polynomial rates is possible.

\begin{example}[Periodised symmetric multivariate Laplace distribution] \normalfont
Consider the symmetric $d$-dimensional multivariate Laplace distribution on $\R^d$, usually defined via its characteristic function $\frac{1}{1+|\xi|^2/2}$, $\xi \in \R^d$. This distribution has a density function of the form (\cite[p. 235]{Laplace})
$$f_d(y) = \frac{2}{(2\pi)^{d/2}} \left( \frac{|y|^2}{2} \right)^{\frac{2-d}{2}} K_{(2-d)/2}(\sqrt{2}|y|),\qquad y \in \R^d,$$
where $K_\lambda$ is the modified Bessel function of the second kind, having integral representation $K_\lambda(x) = \int_0^\infty e^{-x \cosh t} \cosh (\lambda t) dt$.
Consider the periodized version of this density
$$f_d^{per}(x) = \sum_{k\in\Z^d} f_d(x+k), \qquad x \in(0,1]^d,$$
which defines a strictly positive density function on the torus $\T^d$ with Fourier series
\begin{align*}
\hat{f}_d^{per}(k) = F_{\T^d}[f_d^{per}](k) = \int_{\T^d} f_d^{per}(x) e^{-2\pi i k \cdot x} dx =  F_{\R^d}[f_d](k) = \frac{1}{1+2\pi^2|k|^2}, \qquad k \in \Z^d,
\end{align*}
where the relation with the Fourier transform of the unperiodized density $f_d$ follows from~\cite[Thm. 8.35]{folland_1}. For any $m\geq 1$, consider the $m$-fold convolution $\phi_m = f_d^{per} \ast \dots \ast f_d^{per}$, which defines a positive probability density function on $\T^d$ with Fourier transform
$$\hat{\phi}_m(k) = \frac{1}{(1+2\pi^2|k|^2)^m}.$$
Then $\phi_m \in H^\beta$ for any $\beta<2m-d/2$ and satisfies the deconvolution condition \eqref{decon} with $\zeta = 2m$.

Applying Theorem \ref{thm:deconvolution} with $W_0 \in H^{\alpha+1} \cap W^{2,\infty}$ and $\beta = 2m-d/2-\eta$ for $\eta>0$ arbitrarily small, we obtain contraction rate 
$$\Pi(W\in E_{K_N}: \|W-W_0\|_{L^2} \geq LN^{-\theta'}  |D_N) \to^{P_{W_0}} 0$$
as $N\to\infty$, where $\theta'$ is any constant satisfying
$$\theta' < \frac{(\alpha+1+2m-d/2)\frac{2m-d/2-2}{2m-d/2} - 3m}{2(\alpha+1)+4m}$$
(which can get arbitrarily close to the upper bound by taking $\eta>0$ small enough). The right-hand side is strictly positive if and only if
$$\alpha+1 > \frac{(2m-d/2)(m+d/2+2)}{(2m-d/2-2)},$$
in which case we can take $\theta'>0$. Thus for any fixed $m$, one can consistently recover $W_0 \in \mH$ 
with polynomial rates for $\alpha$ large enough.
\end{example}

\section{Proofs}

\subsection{Proof of Theorem \ref{thm:GP_forward}: forward contraction rate}

We verify the conditions of \cite[Thm. 2.2.2]{Nickl_2023}, namely \cite[Cond. 2.1.1]{Nickl_2023} 
with regularization space $\mR = \mathcal{W} \hookrightarrow W^{2,\infty}$ on which the prior concentrates (in view of (\ref{suppw}) and the continuous imbedding $C^2 \hookrightarrow W^{2,\infty}$, e.g. \cite[p.371]{GN16}). This consists of showing (i) boundedness of $\rho_W$ on the domain $\mX$ and (ii) Lipschitz continuity of the forward map $W \mapsto \rho_W: H^{-\beta}(\T^d) \to L^2(\mX)$, uniformly over $\|\cdot\|_{W^{2,\infty}}$-balls of radius $M>0$. While that theorem requires that $W_0 \in \mH \cap \mathcal{W}$, it is straightforward to modify the proof to require only $W_0 \in \mathcal{W}$ along with a suitably good approximation sequence $W_{0,N}$, now using (2.22) in \cite{Nickl_2023} with `projected' RKHS $\mH \cap \mW_N$ in place of $\mH$; see also \cite[Exer. 2.4.3]{Nickl_2023}.

(i) \textit{Boundedness}: using the Sobolev embedding theorem $H^{d/2+\eps}(\T^d) \hookrightarrow C(\T^d)$ for any $\eps>0$ and the $L^{\infty}([0,T] ; H^{k+1}(\T^d))$ regularity estimate \eqref{e:regularity_space} from Theorem \ref{thm:PDE_lin} with $k=d-1$, we have for some $c>0$,
\begin{equation}\label{eq:sup_norm}
\begin{split}
\sup_{W: \|W\|_{\mR} \leq M} \|\rho_{W}\|_{L^\infty(\mX)} & \leq \sup_{W: \|W\|_{W^{2,\infty}} \leq cM} \sup_{0\leq t \leq T} \|\rho_{W}(t)\|_{H^{d/2+\eps}} \\
& \leq \sup_{W: \|W\|_{W^{2,\infty}} \leq cM} \|\rho_{W}\|_{L^{\infty}([0,T] ; H^{d}(\T^d))} \leq C_{M,d,T,\|\phi_0\|_{H^{d}}},
\end{split}
\end{equation}
since $\phi\in H^\beta$ with $\beta \geq 3+d$. This verifies (2.3) in~ \cite[Cond. 2.1.1]{Nickl_2023}.

(ii) \textit{Lipschitz continuity of the forward map}: we control the distance between solutions to the McKean-Vlasov equation \eqref{PDE2} for different potentials via the relative entropy/Kullback-Leibler divergence of the marginal densities. We first require the following estimate from~\cite{Lacker_LeFlem_2023, BRS_2016}, adapted to the torus.

\begin{lemma}\label{lem:bogachev}
Let $W_i \in W^{2,\infty} (\T^d), \, i=1,2$, $\phi \in H^{3+d}(\T^d) \cap \cP(\T^d)$, and consider the solutions $\rho_i , \, i=1,2$, of the corresponding Fokker-Planck equations 
\begin{equation}\label{e:PDE_comparison}
\begin{split}
\frac{\partial \rho_i}{\partial t} & =  \Delta \rho_i + \nabla \cdot (\rho_i (\nabla W_i \ast \rho_i) ),
          \\ \rho_i(0,\cdot)  & =  \phi. 
          \end{split}
\end{equation}
Then for all $t>0$,
\begin{align*}
 H(\rho_1(t) | \rho_2(t)) &\equiv \int_{\T^d} \rho_1(t,x) \log \frac{\rho_1(t,x)}{\rho_2(t,x)} dx  \\
& \leq \frac{1}{2} \int_0^t \int_{\T^d} |\nabla W_1 \ast \rho_1(u) - \nabla W_2 \ast \rho_2(u)|^2 \rho_1(u) \, dx du.
\end{align*}
\end{lemma}

\begin{proof}
The estimate follows from~\cite[Lem. 3.1]{Lacker_LeFlem_2023}, which follows from the arguments in~\cite[Lemma 2.4]{BRS_2016}. As remarked in Sec.~2.4 in \cite{Lacker_LeFlem_2023}, the proof adapts without change to hold for Fokker-Planck equations posed on $\T^d$ instead of $\R^d$. We now check the conditions (H1)-(H3) of~\cite[Lem. 3.1]{Lacker_LeFlem_2023}. 

Under our assumptions on the interaction potential and on the initial condition $\phi$, from Theorem~\ref{thm:MV_well_posed} below it follows that the Fokker-Planck equations~\eqref{e:PDE_comparison} have a unique, strictly positive solution $\rho_i(t) \in \cP(\T^d) \cap C^2(\T^d)$ for all $t >0$. 
By standard convolution inequalities, $\nabla W_2 \ast \rho_2$ is bounded, therefore (H1) is satisfied. This implies, in particular, that the weighted $L^p$ spaces in the proof of~\cite{Lacker_LeFlem_2023}, Lemma 3.1 can be replaced by the unweighted spaces $L^p(\T^d)$.  Furthermore, since $W_i \in W^{2,\infty} (\T^d)$ and $\rho_i \in C^1 ([0,T]; C(\T^d))$ by Theorem~\ref{thm:PDE_lin}, together with Young's inequality for convolutions, we have that $\nabla W_1 \ast \rho_1 \in L^2((0,T) ; L^p (\T^d ))$ for $p > d+2$ (in fact, we can take $p=\infty$),  $\nabla W_2 \ast \rho_2 \in L^2((0,T) ; L^2 (\T^d))$ and $\nabla W_1 \ast \rho_1 \in L^{\infty}((0,T) ; L^2 (\T^d))$, so that (H2)-(H3) are satisfied.
\end{proof}

\begin{lemma}\label{lem:stability_rel_ent}
Let $W_i \in W^{2,\infty} (\T^d)$, $\phi \in H^{3+d}(\T^d) \cap \cP(\T^d)$ and consider the corresponding solutions $\rho_i = \rho_{W_i}$ to \eqref{PDE2}. If $\|W_1\|_{W^{2,\infty}} \leq M$, then for all $t \in [0,T]$,
\begin{equation*}
\int_0^t H(\rho_1(u) | \rho_2(u)) \, du \leq C \int_0^t\| \nabla (W_1 - W_2) \ast \rho_2(s) \|^2_{L^2(\T^d)}ds, 
\end{equation*}
where $C = C(d,T,M,\|\phi\|_{H^{d}})$.
\end{lemma}

\begin{proof}
Arguing as in \eqref{eq:sup_norm} and using Theorem \ref{thm:PDE_lin},
\begin{equation}\label{e:sup_bound}
\|\rho_1\|_{L^\infty(\mX)} \leq C(d) \|\rho_1\|_{L^\infty([0,T];H^{d}(\T^d))} \leq C_0(d,T,M,\|\phi\|_{H^{d}}).
\end{equation}
Using Lemma \ref{lem:bogachev},  Young's convolution inequality and Pinsker's inequality (e.g., \cite[Prop.~6.1.7]{GN16}) we have for all $0 \leq t \leq T$:
\begin{eqnarray*}
H(\rho_1(t) | \rho_2(t))  &\le& \frac{1}{2} \|\rho_1\|_{L^\infty(\mX)} \int_0^t \int_{\T^d}  \left| \nabla W_1 \ast \rho_1(u) - \nabla W_2 \ast \rho_2(u) \right|^2 \,  dx du \\
& \leq & C_0 \int_0^t \int_{\T^d}  \left| \nabla W_1 \ast (\rho_1(u) - \rho_2(u)) \right|^2 \, dx du \\
& & \quad + C_0 \int_0^t \int_{\T^d}  \left| \nabla (W_1-W_2) \ast \rho_2(u) \right|^2 \, dx du \\            
& \leq & C_0 \|\nabla W_1 \|^2_{L^2(\T^d)} \int_0^t \|\rho_1(u) - \rho_2(u) \|^2_{L^1(\T^d)} \, du  \\
& & \quad + C_0 \| \nabla (W_1  - W_2) \ast \rho_2 \|_{L^2([0,t] \times \T^d)}^2 \\             
& \leq & 2 C_0 \|W_1 \|^2_{H^1} \int_0^t H(\rho_1(u) | \rho_2(u)) \, du + C_0 \| \nabla (W_1  - W_2) \ast \rho_2 \|_{L^2([0,t] \times \T^d)}^2.
\end{eqnarray*}
Setting $\eta (t) = \int_0^t H(\rho_1(u) | \rho_2(u)) \, du$, the above estimate implies for $0\le t \le T$
$$
\dot{\eta}(t) \leq 2 C_0 M^2 \eta(t) + C_0 \| \nabla (W_1  - W_2) \ast \rho_2 \|_{L^2([0,T] \times \T^d)}^2.
$$
The result now follows using the differential form of Gronwall's inequality \cite[Sec. B.2]{evans}.
\end{proof}

The second part of the following lemma shows that the map $W \mapsto \rho_W$ is $\beta$-smoothing, suggesting a mildly ill-posed inverse problem. We have not attempted to optimise $\beta$ here but believe it close to being sharp.

\begin{lemma}\label{lem:forward_lip}
Let $W_i \in W^{2,\infty} (\T^d)$, $\phi \in H^{3+d}(\T^d) \cap \cP(\T^d)$ and consider the corresponding solutions $\rho_{W_i}$ to \eqref{PDE2}. Suppose $\|W_i\|_{W^{2,\infty}} \leq M$ for $i=1,2$.
\begin{itemize}
\item[(i)] Then for some constant $C = C(d,T,M,\|\phi\|_{H^{d}})$,
\begin{align*}
\|\rho_{W_1} - \rho_{W_2}\|_{L^2(\mX, \lambda)}^2 \leq C \int_0^T \|\nabla (W_1-W_2) \ast \rho_{W_2}(t)\|_{L^2(\T^d)}^2 dt.
\end{align*}
\item[(ii)]
If in addition $\phi \in H^\beta(\T^d)$, where $\beta\geq 3+d$ is an even integer, then
\begin{align}\label{eq:metric}
\|\rho_{W_1} - \rho_{W_2}\|_{L^2(\mX, \lambda)} \leq C \|W_1-W_2\|_{H^{-\beta}}
\end{align}
for some constant $C(d,\beta,T,M,\|\phi\|_{H^\beta}).$
\end{itemize}
\end{lemma}

\begin{proof}
(i) For any densities $p,q$, recall that the Hellinger distance $h(p,q) = \int (\sqrt{p}-\sqrt{q})^2$ satisfies $h^2(p,q) \leq H(p|q)$  (e.g. \cite[Lemma B.1]{GvdV17}). Moreover, writing $(p-q)^2 = (\sqrt{p}+\sqrt{q})^2(\sqrt{p}-\sqrt{q})^2$ yields $\|p-q\|_{L^2}^2 \leq 4\max (\|p\|_\infty,\|q\|_\infty)h^2(p,q)$. Using these facts and that $\rho_{W_i}(t)$ are probability densities for all $t\geq 0$,
\begin{align*}
\frac{1}{T} \int_0^T \|\rho_{W_1}(t,\cdot) - \rho_{W_2}(t,\cdot)\|_{L^2(\T^d)}^2 dt \leq \frac{4}{T} \max_{i=1,2} \|\rho_{W_i}\|_{L^\infty(\mX)} \int_0^T H(\rho_{W_1}(t,\cdot)|\rho_{W_2}(t,\cdot)) dt. 
\end{align*}
Using \eqref{e:sup_bound} and Lemma \ref{lem:stability_rel_ent} then gives the result.

(ii) Let $e_m(x) = e^{2\pi i m \cdot x}$, $m \in \Z^d$, denote the usual trigonometric polynomials. Since $\langle f(x-\cdot),e_m \rangle_{L^2(\T^d)} = - e_{-m}(x) \langle f,e_{-m} \rangle_{L^2(\T^d)}$, Parseval's theorem yields
\begin{align*}
\left| \partial_{x_i} (W_1-W_2) \ast \rho_{W_2}(t,\cdot)(x) \right|^2 
& \leq 4\pi^2 \left( \sum_{m\in\Z^d} |m_i| |\langle W_1-W_2,e_m \rangle_{L^2}| |\langle \rho_{W_2}(t,\cdot),e_m \rangle_{L^2}| \right)^2 \\
& \leq 4\pi^2 \sum_{m\in\Z^d} |m_i|^2 (1+|m|^2)^{-1-\beta} |\langle W_1-W_2,e_m \rangle_{L^2}|^2 \\
& \qquad \times \sum_{m\in\Z^d}  (1+|m|^2)^{1+\beta} |\langle \rho_{W_2}(t,\cdot),e_m \rangle_{L^2}|^2
\end{align*}
for any $x\in \T^d$. This implies that
\begin{align*}
|\nabla (W_1-W_2) \ast \rho_{W_2}(t,\cdot) (x)|^2 \leq 4\pi^2 \|\rho_{W_2}(t,\cdot)\|_{H^{1+\beta}}^2 \|W_1-W_2\|_{H^{-\beta}}^2.
\end{align*}
By the regularity estimate \eqref{e:regularity_space} in Theorem \ref{thm:PDE_lin}, $\|\rho_{W_2}\|_{L^2([0,T];H^{\beta+1}(\T^d))} \leq C(d,\beta,T,M,\|\phi\|_{H^\beta})<\infty$. Integrating the second last display in space and time, and substituting this into the conclusion of (i) then implies
\begin{align*}
\|\rho_{W_1}-\rho_{W_2}\|_{L^2(\mX)}^2 & \leq C \int_0^T \|\rho_{W_2}(t,\cdot)\|_{H^{\beta+1}}^2 dt ~ \|W_1-W_2\|_{H^{-\beta}}^2  \leq C' \|W_1-W_2\|_{H^{-\beta}}^2.
\end{align*}
\end{proof}

The estimate given in Lemma \ref{lem:forward_lip}(ii) is uniform over $\|\cdot\|_{W^{2,\infty}}$-balls for $W$, which verifies the Lipschitz estimate (2.4) in \cite[Condition 2.1.1]{Nickl_2023} with $\kappa = \beta$. We may thus apply \cite[Theorem 2.2.2]{Nickl_2023}, which completes the proof of Theorem \ref{thm:GP_forward}.

\subsection{Proof of Theorems \ref{stabest} and \ref{thm:deconvolution}: stability estimate and deconvolution}

\begin{proof}[Proof of Theorem \ref{stabest}]
Let $\L_{W,n}\psi :=\Delta \psi  + \nabla \cdot (\psi\nabla W \ast \rho_{n-1}),$ for smooth $\psi$, denote the operator defined in the linearized McKean-Vlasov equation \eqref{e:lin_MV}, $\partial \rho_n/\partial t =\L_{W,n}\rho_{n}$ with initial condition $\phi$. We will prove the result for the linearized solutions before passing to the limit $n\to\infty$. Further define the linear parabolic operator
$\mathscr P_{W,n}=(\partial/\partial t - \mathscr L_{W,n})$, which satisfies $\mathscr P_{W,n} \rho_{W,n}=0$ where $\rho_{W,n}=\rho_n$ is a solution to \eqref{e:lin_MV}, see Theorem \ref{thm:MV_well_posed}.

\smallskip

\textit{Step 1: Reduction to deconvolution with the particle density}.
Let $\rho_W, \rho_{W_0}$ be two solutions of (\ref{PDE2}) for distinct $W,W_0$, and denote by $\rho_{W,n},\rho_{W_0,n}$ their $n$-approximated linearised solutions. As the $\mathscr P_{W,n}$ operators are linear, we can write
\begin{align}\label{pseu0}
\mathscr P_{W,n} (\rho_{W,n}- \rho_{W_0,n}) &=0 - \frac{\partial}{\partial t} \rho_{W_0,n} + \mathscr L_{W,n} \rho_{W_0,n} \notag  \\
&=  (\mathscr L_{W,n}-\mathscr L_{W_0,n})\rho_{W_0,n} \notag \\
&= \nabla \cdot (\rho_{W_0,n} \nabla (W \ast \rho_{W,n-1} - W_0 \ast \rho_{W_0,n-1})).
\end{align}
The densities $\rho_{W_0,n}$ are strictly positive throughout $[0,T] \times \T^d$ in view of (\ref{harnlb}). Hence by standard theory for periodic elliptic equations, the operator $L_{0,n,t}\psi:=\nabla \cdot (\rho_{W_0,n}(t) \nabla \psi)$ is invertible on $L^2_0(\T^d):=L^2(\T^d) \cap \{\int_{\T^d} f =0\}$ for every $t$: if we select the integral zero solution $u = L_{0,n,t}^{-1}h$ to the elliptic equation $L_{0,n,t } u = h$, we have a uniform in $n, t$ elliptic regularity estimate (proved just as on~\cite[p.127-128]{Nickl_2023}),
$$\|L_{0,n,t}^{-1}h\|_{L^2} \le c \|h\|_{H^{-1}},~~ h \in L^2_0(\T^d), $$
with $c$ depending only on the (universal) constant in the Poincar\'e inequality on $\T^d$ as well as the lower bound $\phi_{min}$ from (\ref{harnlb}). We also notice that for any $t$,
\begin{align*}
\int_{\T^d} \mathscr P_{W,n} (\rho_{W,n}- \rho_{W_0,n} ) = \int_{\T^d} (\L_{W, n}- \L_{W_0,n}) \rho_{W_0,n} =0
\end{align*}
as both $\L_{W_n}$ and $\L_{W_0,n}$ are in divergence form. Therefore we can apply the inverse of $L_{0,n,t}$ to (\ref{pseu0}) and obtain
$$W \ast \rho_{W,n-1} - W_0 \ast \rho_{W_0,n-1} = L_{0,n,t}^{-1}[\mathscr P_{W,n} (\rho_{W,n}- \rho_{W_0,n})], ~~\text{ on } [0,T] \times \T^d,$$
where we note that the left hand side integrates to zero since $\int W = \int W_0, \int \rho_{W_0, n-1}=\int\rho_{W, n-1} =1$. Rearranging, we obtain the identity
\begin{equation*}
(W-W_0) \ast \rho_{W_0,n-1} = L_{0,n,t}^{-1}[\mathscr P_{W,n} (\rho_{W,n}- \rho_{W_0,n})] + W \ast (\rho_{W_0,n-1} - \rho_{W,n-1}) 
\end{equation*}
on $[0,T] \times \T^d$. We can now fix any time horizon $0<t_0 \le T$ and estimate
\begin{align}\label{nlinbd}
& \int_0^{t_0} \|(W-W_0) \ast \rho_{W_0,n-1}(t)\|^2_{L^2(\T^d)}dt \notag \\
& \lesssim \int_0^{t_0} \|L_{0,n,t}^{-1}[\mathscr P_{W,n} (\rho_{W,n}- \rho_{W_0,n})(t)]\|_{L^2(\T^d)}^2dt \notag \\
& \qquad + \int_0^{t_0} \|W \ast (\rho_{W_0,n-1} - \rho_{W,n-1})(t)\|_{L^2(\T^d)}^2 dt \notag \\
&\lesssim \int_0^{t_0} \|(\partial/\partial t - \L_{W,n}) (\rho_{W,n}- \rho_{W_0,n})(t)\|_{H^{-1}(\T^d)}^2dt \notag \\
& \qquad  + \int_0^{t_0} \|\rho_{W_0,n-1} - \rho_{W,n-1}(t)\|_{L^2(\T^d)}^2 dt \notag \\
&\lesssim \|\rho_{W,n}- \rho_{W_0,n}\|^2_{H^1([0,t_0], H^{-1})} + \|\rho_{W,n}- \rho_{W_0,n}\|^2_{L^2([0,t_0], H^1)},
\end{align}
where we have used the triangle inequality and the imbedding $H^1 \subset L^2$. We can take limits in this inequality: this is clear for the left hand side by Lemma \ref{fpconv} and since convolution with $W-W_0$ is uniformly Lipschitz on $L^2(\T^d)$. For the right hand side of (\ref{nlinbd}) we notice that by Theorem \ref{thm:PDE_lin} with $\beta=4$ and Rellich's compactness theorem (p.305 in \cite{folland_1}), the sequences $\{\rho_{W,n} - \rho_{W_0,n}: n \in \mathbb N\}$ are relatively compact and hence converge in $H^1([0,t_0], H^{-1})$ and $L^2([0,t_0], H^1)$ along a subsequence to their unique limit, which in view of Lemma \ref{fpconv} equals $\rho_{W} - \rho_{W_0}$. As the constants implied by $\lesssim$ in (\ref{nlinbd}) are $n$-independent, we deduce for any $0<t_0\le T$ that
\begin{equation}\label{nlinbdfin}
\begin{split}
& \int_0^{t_0} \|(W-W_0) \ast \rho_{W_0}(t)\|^2_{L^2(\T^d)}dt \\
& \qquad \quad \le C\big( \|\rho_{W}- \rho_{W_0}\|^2_{H^1([0,t_0], H^{-1})} + \|\rho_{W}- \rho_{W_0}\|^2_{L^2([0,t_0], H^1)}\big)
\end{split}
\end{equation}
for a constant $C=C(T,d, M, \phi_{min}, \|\phi\|_{H^\beta})>0$.

\smallskip

\textit{Step 2. Deconvolving at small times.}
Let $W_{0,K} = \sum_{k\in\Z^d: |k| \leq K} \langle W_0 , e_k\rangle_{L^2}  e_k$ denote the Fourier projection of $W_0\in H^{\alpha+1}$ onto the first $K$ frequencies. Expanding the $L^2$-norm in terms of the Fourier coefficients using Parseval's theorem and since $\int W = \int W_0$, we have 
\begin{equation}\label{trivdec}
\begin{split}
\|W - W_0\|_{L^2}^2 &\le \|W_0\|_{H^{\alpha+1}}^2 K^{-2(\alpha+1)} + \|W-W_{0,K}\|^2_{L^2}  \\
 &\lesssim K^{-2\alpha-2} + \sum_{0<|k| \le K} |\hat W(k)-\hat W_0(k)|^2 \frac{|\hat \rho_{W_0,n}(t,k)|^2}{|\hat \rho_{W_0,n}(t, k)|^2} \\
 &\lesssim K^{-2\alpha-2}+\sup_{0<|k|\le K} \frac{1}{|\hat \rho_{W_0,n}(t, k)|^2} \|(W-W_0) \ast \rho_{W_0,n}(t, \cdot)\|_{L^2}^2 
\end{split}
\end{equation}
For $t>0$, define the stability constant
\begin{equation}\label{illness}
\iota_{K, n}(t) \equiv \sup_{0<|k|\le K} \frac{1}{|\hat \rho_{W_0,n}(t, k)|^2},
\end{equation}
which we now upper bound. By the second part of Theorem \ref{thm:PDE_lin} we know $$\big\|\frac{d\rho_{W_0,n}}{dt}\big\|_{L^\infty([0,T];H^{\beta})} \leq C(d,T,\beta,M,\|\phi\|_{H^\beta}),$$ and hence $t \mapsto \rho_{W_0,n}(t)$ is Lipschitz as a map from $[0,T] \to L^1(\T^d)$ which implies $\|\rho_{W_0,n}(t) - \phi\|_{L^1(\T^d)} \leq Ct$. But since $L^1$-norms bound the Fourier coefficients uniformly, for any $k\in \Z^d$,
$$|\hat \rho_{W_0,n}(k)| \ge |\hat \phi(k)| - |\hat \rho_{W_0,n}(k) - \hat \phi(k)| \ge c|k|^{-\zeta} - Ct.$$
Thus for any $0<t\leq t_0 = \tfrac{c}{2C} K^{-\zeta}$, the stability constant in \eqref{illness} satisfies $\iota_{K,n}(t) \leq \frac{c^2}{4} K^{2\zeta}$.

Integrating the inequality \eqref{trivdec} in time then gives
$$ \int_0^{t_0} \|W - W_0\|_{L^2}^2 dt \lesssim t_0 K^{-2\alpha-2} + K^{2\zeta} \int_0^{t_0}\|(W-W_0) \ast \rho_{W_0,n}(t, \cdot)\|_{L^2}^2dt$$
and hence by (\ref{nlinbd}),
\begin{equation*}
\|W - W_0\|_{L^2}^2 \lesssim K^{-2\alpha-2} + t_0^{-1} K^{2\zeta}(\|\rho_{W}- \rho_{W_0}\|^2_{H^1([0,t_0], H^{-1})} + \|\rho_{W}- \rho_{W_0}\|^2_{L^2([0,t_0], H^1)}),
\end{equation*}
completing the proof.
\end{proof}

\begin{proof}[Proof of Theorem \ref{thm:deconvolution}]
The proof follows from combining the contraction rate for the forward map from Example \ref{ex:Fourier_series} with the stability estimate from Theorem \ref{stabest}. By Example \ref{ex:Fourier_series}, the posterior concentrates on the set
$$\overline{\mathcal{W}}_N = \{W\in E_{K_N}:  \|\rho_W-\rho_{W_0}\|_{L^2(\mX)} \le L\delta_N , ~ \|W\|_{W^{2,\infty}} \leq M\},$$
where $\delta_N = N^{-\frac{\alpha+1+\beta}{2(\alpha+1)+2\beta+d}}$, so it suffices to prove the result for $W\in \overline{\mathcal{W}}_N$. Using the interpolation inequality~[Eqn. (A.5)]\cite{Nickl_2023}, $\|u\|_{H^1(\T^d)} \lesssim \|u\|_{L^2(\T^d)}^{1-1/\gamma}\|u\|_{H^\gamma(\T^d)}^{1/\gamma}$ for any $\gamma>1$. By \eqref{e:regularity_space} of Theorem \ref{thm:PDE_lin}, $\|\rho_W\|_{L^2([0,T];H^{\beta+1})} \leq C$ for any $W\in \overline{\mathcal{W}}_N$ as well as $W_0 \in H^{\alpha+1} \subset \mathcal W^{2,\infty}$ via the Sobolev imbedding, and hence applying the interpolation inequality with $\gamma=\beta+1$ and using H\"older's inequality yields
\begin{align*}
\|\rho_W - \rho_{W_0}\|_{L^2([0,T];H^1)}^2 & \lesssim \left( \int_0^T \|\rho_W(t) - \rho_{W_0}(t)\|_{L^2(\T^d)}^2 dt \right)^{\frac{\gamma-1}{\gamma}}  = \|\rho_W - \rho_{W_0}\|_{L^2(\mX)}^{\frac{2\beta}{\beta+1}}.
\end{align*}
Turning to the time regularity, using~\cite[Ch.~4, Prop. 2.1]{Lions_vol2}, we have the interpolation inequality $$\|v\|_{H^1([0,T];L^2)} \lesssim \|v\|_{L^2([0,T];L^2)}^{(m-1)/m} \|v\|_{H^m([0,T];L^2)}^{1/m}$$
for any integer $m\geq 1$. But by \eqref{e:regularity_time} of Theorem \ref{thm:PDE_lin}, $\|\frac{d^\ell\rho_W}{dt^\ell}\|_{L^2([0,T];H^1)} \leq C$ for all $\ell =0,1,\dots,\beta/2$ and hence we may take $m=\beta/2$ to obtain
\begin{align*}
\|\rho_W - \rho_{W_0}\|_{H^1([0,T];H^{-1})} \lesssim \|\rho_W - \rho_{W_0}\|_{L^2([0,T];L^2)}^\frac{m-1}{m} = \|\rho_W - \rho_{W_0}\|_{L^2(\mX)}^\frac{\beta-2}{\beta}.
\end{align*}
Apply Theorem \ref{stabest} with $t_0 \simeq K_N^{-\zeta}$ gives that for any $W \in \overline{\mathcal{W}}_N \subset E_{K_N}$,
\begin{align*}
\|W-W_0\|_{L^2(\T^d)}^2 & \lesssim K_N^{-2\alpha-2} + K_N^{3\zeta} \left( \delta_N^{\frac{2(\beta-2)}{\beta}} + \delta_N^\frac{2\beta}{\beta+1} \right) \lesssim  K_N^{-2\alpha-2} + K_N^{3\zeta} \delta_N^{\frac{2(\beta-2)}{\beta}}.
\end{align*}
Substituting in $K_N \simeq N^{\frac{1}{2(\alpha+1)+2\beta+d}}$ and $\delta_N$ shows that the second term dominates for all $\alpha>0$ and $\beta \geq 3+d$ since $\zeta > \beta + d/2$. The convergence of the posterior mean vector now follows from arguments as in~\cite[Theorem 2.3.2]{Nickl_2023} and the details are left to the reader.
\end{proof}

\subsection{Regularity properties for the McKean-Vlasov PDE}

In this section, we collect properties and regularity estimates for the nonlinear parabolic McKean-Vlasov PDE \eqref{PDE2} driving the dynamics of our model, as well as its linearization. Recall that we work on the torus $\T^d = \R^d/\Z^d= (0,1]^d$ and that $\cP = \cP(\T^d)$ denotes the set of probability densities (with respect to Lebesgue measure) on $\T^d$. The basic existence and uniqueness of classical solutions to \eqref{PDE2} follows from~\cite{CGPS2020,Gvalani_thesis}; see also~\cite[Thm. 2]{guillin2023uniform}.

\begin{theorem}(\cite[Theorem 2.2]{CGPS2020})\label{thm:MV_well_posed}
If $W \in W^{2, \infty}(\T^d)$ and $\phi \in H^{3+d}(\T^d) \cap \cP(\T^d)$, then there exists a classical solution $\rho:[0,T]\times \T^d\to[0,\infty)$ such that $\rho (t, \cdot) \in \cP(\T^d) \cap C^2(\T^d)$ for all $t >0$.
\end{theorem}

Similar to \cite{sznitman1991, Chazelle_al2017b,CGPS2020}, we approximate the McKean-Vlasov equation \eqref{PDE2} by a sequence of linear parabolic PDEs. For $n=1,2,\dots$, consider the sequence of linear equations on $\T^d$
\begin{align}\label{e:lin_MV}
\frac{\partial \rho_n}{\partial t} &=  \Delta \rho_n + \nabla \cdot (\rho_n \nabla W \ast \rho_{n-1})
=: \mathscr L_{W,n} \rho_n, \\
\rho_n(0,\cdot) & = \phi \notag,
\end{align}
where $\phi \in H^{\beta}(\T^d) \cap \cP(\T^d)$ and we take a smooth time-independent initialization $\rho_0 \in C^{\infty}((\T^d)) \cap \cP(\T^d)$. One then shows as in \cite{sznitman1991, Chazelle_al2017b,CGPS2020}, see specifically \cite[Sec. 3.2.2]{Gvalani_thesis}:

\begin{lemma} \label{fpconv}
Let $W \in W^{2, \infty}(\T^d)$ and $\phi \in H^{3+d}(\T^d) \cap \cP(\T^d)$. As $k \to \infty$ we have (if necessary along a subsequence) $\rho_{k} \to \rho$ in $L^2([0,T]; H^1(\T^d))$ as well as in $H^1([0,T]; H^{-1}(\T^d))$. 
\end{lemma}

We now establish quantitative regularity estimates which hold uniformly over potentials $W$ of bounded norm, which was crucial in our proofs above. It builds on the results of \cite{Chazelle_al2017b,Gvalani_thesis}. We remark that the smoothness of $\rho_W$ is driven by the regularity of the initial condition $\phi$, while the regularity of $W$ is less crucial as it factors into the estimates only after convolution with $\rho_{n-1}$, which is a smoothing operation.

\begin{theorem}\label{thm:PDE_lin}
Let $\beta \ge 3+d$ be an integer and suppose $W \in W^{2,\infty}(\T^d)$ satisfies $\|W\|_{W^{2,\infty}} \le M$. Let further $\phi \in H^\beta(\T^d) \cap \cP(\T^d)$ be a strictly positive probability density such that $\inf_{x}\phi(x) \ge \phi_{min}>0$ and let $\rho_0 \in C^{\infty}(\T^d) \cap \cP(\T^d)$ be strictly positive. Then for every $n \in \mathbb{N}$, the PDE \eqref{e:lin_MV} has a unique solution $\rho_n \in C^{1,2}([0,T] \times \T^d)$, which is strictly positive
\begin{equation}\label{harnlb}
\inf_{x \in \T^d, t \in [0,T]}\rho_{n}(t,x) \ge \lambda (\phi_{min}, M,T)>0,
\end{equation}
for a constant $\lambda$ that depends only on $\phi_{min},M,T$, and satisfies $\int_{\T^d} \rho_n(t,x) \, dx =1$ for all $t\in[0,T]$. For $k=0,1,\dots,\beta-1$, we have the following regularity estimate in space:
\begin{equation}\label{e:regularity_space}
\|\rho_n \|_{L^2([0,T] ; H^{k+2}(\T^d))} + \| \rho_n \|_{L^{\infty}([0,T] ; H^{k+1}(\T^d))} \leq C(d,k,T,M,\|\phi\|_{H^{k+1}}).
\end{equation}
Furthermore, if $\beta = 2j$ is an even integer, we have time regularity
$$
 \frac{d^\ell \rho_n}{dt^\ell} \in L^2 ([0,T]; H^{2j-2\ell+1}(\T^d)) \cap L^{\infty} ([0,T] ; H^{2j-2\ell}(\T^d)) \quad \mbox{for} \; 0 \leq \ell\leq j,
$$
with the estimate
\begin{equation}\label{e:regularity_time}
\begin{split}
& \sum_{\ell=0}^{j} \left( \left\|\frac{d^\ell \rho_n}{d t^\ell} \right\|_{L^2([0,T] ; H^{2j - 2\ell +1}(\T^d))} + \left\| \frac{d^\ell \rho_n}{d t^\ell} \right\|_{L^{\infty}([0,T] ; H^{2j -2\ell}(\T^d))} \right) \\
& \qquad \qquad \leq C(d,j,T,M,\|\phi\|_{H^{2j}}, \|\rho_0\|_{H^{2j}}).
\end{split}
\end{equation}


Moreover, all the above estimates hold with $\rho_n$ replaced by $\rho$, the solution to the McKean-Vlasov equation \eqref{PDE2}.
\end{theorem}

\begin{proof}
The existence of a unique solution $\rho_n$ to \eqref{e:lin_MV} that is a probability density follows from standard linear parabolic PDE theory, e.g.~\cite[Ch. 7]{evans}. The basic regularity in time and space follow from~\cite[Theorems 27.2-27.3]{Wloka_1987} (time) and~\cite[Theorem 27.5]{Wloka_1987} (space), see also~\cite[Thm. 2, Thm. 3]{Milani_1999}. In particular, since $\rho_0$ is smooth, for $n=1$  we have a linear parabolic PDE (written in non-divergence form) with smooth coefficients $\nabla W \ast \rho_0 \in C^{\infty}$ and $\Delta W \ast \rho_0 \in C^{\infty}$. Furthermore, the compatibility conditions are automatically satisfied since we have periodic boundary conditions. Therefore, from~\cite[Thm. 2]{Milani_1999} it follows that $\rho_1 \in L^2 ([0,T]; H^{\beta+1}(\T^d)) \cap H^{\frac{\beta +1}{2}} ([0,T] ; L^2(\T^d))$. From this we deduce that we have the same space/time regularity as $\rho_1$ for $\nabla W \ast \rho_{n-1}$ and of $\Delta W \ast \rho_{n-1}$, since $W \in W^{2, \infty}$. Applying again~\cite[Thm. 2]{Milani_1999} we conclude the existence and uniqueness of solutions $\rho_n \in L^2 ([0,T]; H^{\beta+1}(\T^d)) \cap H^{\frac{\beta +1}{2}} ([0,T] ; L^2(\T^d))$ to~\eqref{e:lin_MV}, which, using \cite{Milani_1999}, Equation 1.17, see also \cite{evans} Section 7.1, Theorem 6, justifies taking any derivatives we need in what follows. It remains to show the quantitative regularity estimates, which we show hold uniformly in $n$ and over potentials $\|W\|_{W^{2,\infty}} \leq M$.


We start with space regularity and follow ideas in \cite{Chazelle_al2017b,Gvalani_thesis}. We prove~\eqref{e:regularity_space} by induction on $k$, starting with $k=0$. We multiply \eqref{e:lin_MV} by $\rho_n$, integrate in $x$, do the standard integration by parts, use H\"{o}lder and Young inequalities to deduce
\begin{align*}
\frac{1}{2}\frac{d}{d t} \| \rho_n \|^2_{L^2} + \|\nabla \rho_n \|^2_{L^2} &= \langle \rho_n, \frac{d}{dt} \rho_n - \Delta\rho_n \rangle_{L^2}\\
& = - \langle \nabla \rho_n, \rho_n \nabla W_{n-1} \rangle_{L^2} \\
& \leq \frac{1}{2} \|\nabla \rho_n\|_{L^2}^2 + \frac{1}{2} \|\rho_n\|_{L^2}^2\|\nabla W_{n-1}\|_{L^2}^2,
\end{align*}
where we have used the notation $W_{n-1} := W \ast \rho_{n-1}$. By Young's convolution inequality, $\|\nabla W_{n-1}\|_{L^2} \leq \|\nabla W\|_{L^2} \|\rho_{n-1}\|_{L^1} \leq M$, from which we deduce
\begin{equation}\label{e:diff_inequality}
\frac{1}{2}\frac{d}{d t} \| \rho_n \|^2_{L^2} + \frac{1}{2}\|\nabla \rho_n \|^2_{L^2} \leq \frac{M^2}{2}\|\rho_n\|_{L^2}^2
\end{equation}
for all $t\in[0,T]$, where we recall $\rho_n = \rho_n(t,\cdot)$. Using the differential form of Gronwall's inequality~\cite[Sec. B.2]{evans}, we have
\begin{equation}\label{e:L2_estim}
\| \rho_n(t,\cdot) \|^2_{L^2} \leq e^{M^2T} \|\rho(0,\cdot) \|^2_{L^2} = e^{M^2T} \|\phi\|_{L^2}^2
\end{equation}
for all $t \in [0,T]$. Furthermore, integrating~\eqref{e:diff_inequality} in time and using~\eqref{e:L2_estim},
\begin{equation}\label{e:H1_estim}
\sup_{t \in [0,T]} \|\rho_n(t,\cdot) \|^2_{L^2} + \int_0^T \| \nabla \rho_n(t,\cdot) \|^2_{L^2} \, dt \leq (M^2T e^{M^2T}+1) \|\phi\|^2_{L^2}.
\end{equation}
We proceed in a similar manner with the higher regularity estimates. For the $H^2$-in-space estimate, we multiply \eqref{e:lin_MV} by $\Delta \rho_n$, integrate by parts, use H\"{o}lder's and Young's inequalities 
to deduce that
\begin{eqnarray*}
\frac{1}{2} \frac{d}{d t} \|\nabla \rho_n \|_{L^2}^2 + \| \Delta \rho_n \|^2_{L^2} & = &  -\int_{\T^d} \nabla \cdot \left(\rho_n \nabla W_{n-1} \right) \Delta \rho_n \, dx 
\\ & \leq & \frac{1}{2} \|\Delta \rho_{n} \|^2_{L^2} + \frac{1}{2} \int_{\T^d} \Big| \nabla \cdot \big( \rho_n \nabla W_{n-1} \big) \Big|^2 \, dx
\\ & \leq & \frac{1}{2} \|\Delta \rho_{n} \|^2_{L^2} +  \|\nabla \rho_{n} \cdot\nabla W_{n-1}\|_{L^2}^2 +  \| \rho_n \Delta W_{n-1}  \|_{L^2}^2
\\ & \leq & \frac{1}{2} \|\Delta \rho_{n} \|^2_{L^2} +  \|\nabla W_{n-1}\|_{L^\infty}^2 \|\nabla \rho_{n} \|^2_{L^2} + \|\Delta W_{n-1}\|_{L^\infty}^2 \| \rho_n \|_{L^2}^2.
\end{eqnarray*}
Using \eqref{e:diff_inequality}-\eqref{e:L2_estim} and the fact that $\| \rho_{n-1} \|_{L^1} =1$, we similarly obtain 
$\tfrac{d}{dt}\|\nabla \rho_n \|_{L^2}^2 + \| \Delta \rho_n \|^2_{L^2} \leq C(M,T)\|\phi\|_{L^2}^2.$ Integrating in time and using \eqref{e:L2_estim} and \eqref{e:H1_estim} we get
\begin{equation}\label{e:H2_estim}
\sup_{t \in [0,T]} \|\nabla \rho_n \|^2_{L^2} + c \int_0^T \| \rho_n \|^2_{H^2} \, dt \leq C(T,M)\|\phi\|_{H^1}^2.
\end{equation}
Together,~\eqref{e:H1_estim} and~\eqref{e:H2_estim} establish  \eqref{e:regularity_space} for $k=0$.

For the higher regularity estimates \eqref{e:regularity_space} with integers $1\leq k \leq \beta-1$, we proceed by induction on $k$ as in \cite{Chazelle_al2017b,Gvalani_thesis}. Assume that \eqref{e:regularity_space} holds for $k-1$. We differentiate~\eqref{e:lin_MV} by $\partial_{\nu}$ for some multi-index $\nu$ with $|\nu| = k$, multiply by $\Delta \partial_{\nu} \rho_n$, integrate by parts and use Young's inequality to obtain
\begin{equation}\label{e:Hk-norm}
\begin{split}
\frac{1}{2}\frac{d}{d t} \|\nabla \partial_{\nu} \rho_n \|^2_{L^2} +  \|\Delta \partial_{\nu} \rho_n \|^2_{L^2} & \leq \Big| \int_{\T^d} \Delta \partial_{\nu} \rho_n 
\nabla \cdot (\rho_n \nabla W_{n-1} )  \, dx \Big| \\
& \leq  \frac{1}{2} \|\Delta \partial_{\nu}
 \rho_{n} \|^2_{L^2} + C_d \sum_{i=1}^d  \| \rho_n \big( \partial_{x_i} W \ast \rho_{n-1} \big)  \|^2_{H^{k+1}}.
\end{split}
\end{equation}
But by Lemma \ref{lem:convolution},
$$
\| \rho_n \big( \partial_{x_i} W \ast \rho_{n-1} \big) \|^2_{H^{k+1}} \leq C(d,k,M) \| \rho_n \|_{H^{k+1}} \|\rho_{n-1} \|_{H^k}. 
$$
Substituting this estimate into~\eqref{e:Hk-norm}, summing over all $\nu = |k|$, integrating over $[0,T]$ and using the induction hypothesis,
\begin{align*}
\sup_{t \in [0,T]} \| \rho_n(t) \|^2_{H^{k+1}} + \int_0^T \| \rho_n(t) \|^2_{H^{k+2}}dt &  \leq C(d,k,M) \int_0^T \|\rho_n(t)\|_{H^{k+1}}^2 \|\rho_{n-1}(t)\|_{H^k} dt  \\
& \qquad + \|\phi\|_{H^{k+1}}^2 \\
& \leq C(d,k,M,T,\|\phi\|_{H^{k+1}}, \|\rho_0\|_{H^{k+1}}),
\end{align*}
which establishes \eqref{e:regularity_space} for $k$ as desired. Now since the $\rho_n$ are uniformly bounded in the relevant Hilbert space norms, the Banach-Alaoglu theorem implies that they converge weakly in these Hilbert spaces to their limits along a subsequence. Combined with the inequality $\|h\|_{H} \le \liminf_n \|h_n\|_{H}$ for any weakly convergent sequence $h_n \to h$ in a Hilbert space, we may take limits in the last inequalities, to give the following apriori estimate for the solution to the McKean-Vlasov PDE
\begin{equation}\label{e:Hk-estim-mckean}
\|\rho \|_{L^2([0,T] ; H^{k+2}(\T^d)}) + \|\rho \|_{L^{\infty}([0,T] ; H^{k+1}(\T^d))} \leq C(d,k,M,T,\|\phi\|_{H^{k+1}}, \|\rho_0\|_{H^{k+1}}).
\end{equation}
We now consider regularity in time. For $\beta =2 j$, we have that $\rho_n \in L^2 ([0,T]; H^{2j+1}(\T^d)) \cap H^{j+\frac{1}{2}} ([0,T] ; L^2(\T^d))$ with $\partial^{j-1}_t \rho_n \in C([0,T] ; H^1(\T^d))$, see~\cite[Thm. 2]{Milani_1999}.
To get the quantitative estimates, we follow the proof of~\cite[Theorem 4.3]{Chazelle_al2017b}, proceeding by induction on the number of time-derivatives $\ell$. 
For $\ell=0$, this follows immediately from \eqref{e:regularity_space}. Suppose now we have the time regularity estimate~\eqref{e:regularity_time} for some integer $0 \leq \ell < j$. We differentiate \eqref{e:lin_MV} $\ell$-times in $t$ to obtain, using the notation $\rho^{(\ell)}:=\frac{d^{\ell} \rho}{d t^{\ell}}$,
\begin{eqnarray*}
\rho_n^{(\ell+1)} &=& \Delta \rho_n^{(\ell)} + \nabla \cdot \left(\rho_n \nabla W \ast \rho_{n-1} \right)^{(\ell)}
\\ & = & \Delta \rho_n^{(\ell)} + \sum_{k=0}^{\ell} \binom{\ell}{k} \nabla \cdot \left(\rho_n^{(k)} \nabla W \ast \rho_{n-1}^{(\ell-k)} \right).
\end{eqnarray*}
Taking the $H^{2j - 2 \ell -1}(\T^d)$-norm in space and using Lemma \ref{lem:convolution} since $2j - 2\ell \geq 2$,
\begin{eqnarray*}
\| \rho_n^{(\ell+1)} \|^2_{H^{2j - 2 \ell -1}(\T^d)} & \leq & 2 \| \rho_n^{(\ell)} \|^2_{H^{2j - 2 \ell +1}(\T^d)} + C_\ell \sum_{k=0}^{\ell} \|\rho_n^{(k)} \nabla W \ast \rho_{n-1}^{(\ell-k)}  \|^2_{H^{2j - 2 \ell}(\T^d)}
 \\ & \leq & 
                 2 \| \rho_n^{(\ell)} \|^2_{H^{2j - 2 \ell +1}(\T^d)} + C_{j,M} \sum_{k=0}^{\ell} \|\rho_n^{(k)} \|^2_{H^{2j - 2 \ell}(\T^d)} \| \rho_{n-1}^{(\ell-k)}  \|^2_{H^{2j - 2\ell-1}(\T^d)},
\end{eqnarray*}
where $C$ depends only on $j$ and $\|W\|_{W^{2,\infty}(\T^d)} \leq M$. Integrating over time $t \in [0,T]$ then gives
\begin{eqnarray*}
\| \rho_n^{(\ell+1)} \|^2_{L^2([0,T] ; H^{2j - 2 \ell -1}(\T^d))} & \leq & 2 \| \rho_n^{(\ell)} \|^2_{L^2([0,T] ; H^{2j - 2 \ell +1}(\T^d))}
 \\ & & 
                  + C_{j,M} \sum_{k=0}^{\ell} \|\rho_n^{(k)} \|^2_{L^2([0,T] ; H^{2j - 2 \ell}(\T^d))} \| \rho_{n-1}^{(\ell-k)}  \|^2_{L^\infty([0,T] ; H^{2j - 2 \ell-1}(\T^d))}.
\end{eqnarray*}
Since $0 \leq k,\ell-k \leq \ell$, we may now apply the inductive hypothesis \eqref{e:regularity_time} to conclude that
$$
\| \rho_n^{(\ell+1)} \|^2_{L^2(0,T ; H^{2j- 2\ell -1}(\T^d))}  \leq  C(d,j,T,\|W\|_{W^{2,\infty}},\|\phi\|_{H^{2j}},\|\rho_0\|_{H^{2j}}),
$$
which is the first part of the required estimate for $\ell+1$. Similarly, taking the $H^{2j-2\ell-2}$-norm in space and arguing as above yields
$$
\| \rho_n^{(\ell+1)} \|^2_{L^{\infty}(0,T ; H^{2j - 2 \ell -2}(\T^d))}  \leq  C(d,j,T,\|W\|_{W^{2,\infty}},\|\phi\|_{H^{2j}},\|\rho_0\|_{H^{2j}}),
$$
which proves the desired estimate for $\ell+1$. This proves \eqref{e:regularity_time} for $\ell = 0,\dots,j$ as required, based on the smoothness of the initial condition $\phi \in H^{2j}$. As before, we can then pass to the limit $n \to \infty$ to obtain the estimates for the mean field PDE.


To obtain the lower bound \eqref{harnlb} to the solution to the McKean-Vlasov PDE, we follow the proof of~\cite[Lem. 5.1]{Lacker_LeFlem_2023} that is based on the argument presented in~\cite[Thm. 2]{guillin2023uniform}. Fix $T'\in(0,T]$, let $\rho_t(x)= \rho(t,x)$ denote the solution to the McKean-Vlasov PDE \eqref{PDE2},  and consider the unique strong solution of the 
SDE
\begin{equation*}
d Y_t = \nabla W \ast \rho_{T'-t} (Y_t) \, dt + \sqrt{2} \, dB_t, \quad Y_0 =x.
\end{equation*}
Note that we regard here $\rho_t(x)$ as a given function and not as the law of the process $Y_t$. The generator of $Y_t$ is $\mathcal{L} = \nabla W \ast \rho_{T'-t} (x) \cdot \nabla + \Delta.$  Using It\^{o}'s formula, the PDE \eqref{PDE2} and taking expectations gives for $t\in[0,T']$, as in~\cite[Eqn. 5.4]{Lacker_LeFlem_2023}
\begin{equation}\label{e:expect_harnack}
\E^x \rho_{T'-t} (Y_t) = \rho_{T'}(x) - \E^x \int_0^t \rho_{T'-s} (Y_s) \Delta W \ast \rho_{T'-s} (Y_s) \, ds,
\end{equation}
where $\E^x$ denotes the expectation with respect to the law of the process $Y_t$ starting at $x$. We use now the fact that $\rho_u(x) \geq 0, \, u \in [0,T]$, together with the estimate $\|\Delta W \ast \rho_u \|_{L^{\infty}} \leq \|W\|_{W^{2,\infty}} \|\rho_u\|_{L^1} \leq M$ to deduce that
$$
\E^x \rho_{T'-t} (Y_t) \leq \rho_{T'}(x) + M \int_0^t \E^x \rho_{T'-s} (Y_s)  \, ds.
$$
From the integral form of Gronwall's inequality, it follows that for all $t\in[0,T']$,
$$
\E^x \rho_{T'-t}(Y_t) \leq \rho_{T'}(x) \left(1 + M t e^{M t} \right) \leq \rho_{T'}(x) \left( 1+MTe^{MT} \right).
$$
Setting now $t=T'$ gives that for any $T' \in (0,T]$,
$$
\rho_{T'}(x) \geq \frac{\phi_{min}}{1 + MT e^{MT}},
$$
as required. The same argument also applies to the sequence of PDEs~\eqref{e:lin_MV}. We omit the details.
\end{proof}

\begin{lemma}\label{lem:convolution}
Let $W \in W^{2,\infty}(\T^d)$, $f\in H^k$ and $g \in H^{k-1}$ for integer $k\geq 2$. Then
\begin{equation}\label{e:estim_conv_2}
\|f (\nabla W *g) \|_{H^k} \leq C(k,d,\|W\|_{W^{2,\infty}}) \|f \|_{H^k} \|g \|_{H^{k-1}}. 
\end{equation}
\end{lemma}

\begin{proof}
We can estimate the term on the left-hand-side by a constant multiple of $\|f\|_{H^k} \|\nabla W \ast g\|_{C^k}$ and then use the periodic analogue of  \cite[Lemma 4.3.18]{GN16}.
\end{proof}

%
%

\begin{acks}[Acknowledgments]
The authors would like to thank the associate editor and two anonymous referees for their many helpful remarks and suggestions.
\end{acks}

\begin{funding}
RN was supported by an ERC Advanced Grant (Horizon Europe UKRI G116786) as well as by EPSRC programme grant EP/V026259. GP is partially supported by an ERC-EPSRC Frontier Research Guarantee through Grant No. EP/X038645, ERC Advanced Grant No. 247031 and a Leverhulme Trust Senior Research Fellowship, SRF$\backslash$R1$\backslash$241055.
\end{funding}



\bibliographystyle{imsart-number} 
\bibliography{mybib}       


\end{document}